\documentclass[12pt,reqno]{amsart}
\usepackage{amsmath, amsfonts, amssymb, amsthm,mathrsfs}

\large\normalsize
\theoremstyle{plain}
\newtheorem{theorem}{Theorem}
\newtheorem{corollary}[theorem]{Corollary}
\newtheorem{lemma}[theorem]{Lemma}

\theoremstyle{definition}
\newtheorem{definition}[theorem]{Definition}
\newtheorem{remark}[theorem]{Remark}
\newtheorem{example}[theorem]{Example}

\newcommand{\A}{\ensuremath{\mathbb{A}}}
\newcommand{\R}{\ensuremath{\mathbb{R}}}
\newcommand{\Z}{\ensuremath{\mathbb{Z}}}

\newcommand{\D}{\ensuremath{\mathbb{D}}}

\newcommand{\cF}{\mathcal{F}}

\newcommand{\mat}[4]{\begin{pmatrix}#1&#2\\#3&#4\end{pmatrix}}
\newcommand{\vecc}[2]{\begin{pmatrix}#1\\#2\end{pmatrix}}

\numberwithin{equation}{section}
\numberwithin{theorem}{section}

\begin{document}

\title[Self-Adjoint Proportional Differential Equations]{Second-Order Self-Adjoint Differential Equations Using a Conformable Proportional Derivative}
\author[Anderson]{Douglas R. Anderson} 
\address{Department of Mathematics \\
         Concordia College \\
         Moorhead, MN 56562 USA}
\email{andersod@cord.edu}


\begin{abstract}
In this study, linear second-order conformable differential equations using a proportional derivative are shown to be formally self-adjoint equations with respect to a certain inner product and the associated self-adjoint boundary conditions. Defining a Wronskian, we establish a Lagrange identity and Abel's formula. Several reduction-of-order theorems are given. Solutions of the conformable second-order self-adjoint equation are then shown to be related to corresponding solutions of a first-order Riccati equation and a related quadratic functional and a conformable Picone identity. The first part of the study is concluded with a comprehensive roundabout theorem relating key equivalences among all these results. Subsequently, we establish a Lyapunov inequality, factorizations of the second-order equation, and conclude with a section on boundary value problems and Green's functions.
\end{abstract}

\maketitle\thispagestyle{empty}

%%%%%%%%%%%%%%%%%%%%%%%%%%%%%%%
%   Introduction Section 1    %
%%%%%%%%%%%%%%%%%%%%%%%%%%%%%%%

\section{Introduction}\label{secintro}

We study the linear second-order conformable self-adjoint equation
\[ D^{\alpha}[pD^{\alpha}x](t)+q(t)x(t)=h(t), \]
where $D^{\alpha}$ is a proportional derivative operator. Conformable proportional derivatives have been introduced by Anderson and Ulness \cite{au} to generalize the idea of a proportional derivative controller \cite{ang}. This new conformable derivative operator $D^\alpha$ of order $\alpha\in[0,1]$, where $D^0$ is the identity operator, and $D^1$ is the classical differential operator, will be used to explore conformable self-adjoint equations.

% Definition %

\begin{definition}[Conformable Differential Operator]\label{conformdef}
Let $\alpha\in[0,1]$. A differential operator $D^{\alpha}$ is conformable if and only if $D^0$ is the identity operator and $D^1$ is the classical differential operator. Specifically, $D^{\alpha}$ is conformable if and only if for a differentiable function $f=f(t)$, 
\[ D^0 f(t) = f(t) \quad\text{and}\quad D^1 f(t) = \frac{d}{dt}f(t) = f'(t). \]
\end{definition}

% Remark %

\begin{remark}
In control theory, a proportional-derivative (PD) controller for controller output $u$ at time $t$ with two tuning parameters has the algorithm
\[ u(t) =  \kappa_pE(t) + \kappa_d\frac{d}{dt}E(t) , \]
where $\kappa_p$ is the proportional gain, $\kappa_d$ is the derivative gain, and $E$ is the error between the state variable and the process variable; see \cite{ang}, for example. This is the impetus for the next definition.
\end{remark}

% Definition %

\begin{definition}[A Class of Conformable Derivatives]
Let $\alpha\in[0,1]$, $\mathcal{I}\subseteq\R$, and let the functions $\kappa_0,\kappa_1:[0,1]\times\mathcal{I}\rightarrow[0,\infty)$ be continuous such that
\begin{eqnarray} 
 \lim_{\alpha\rightarrow 0^+} \kappa_1(\alpha,t) = 1, & \displaystyle\lim_{\alpha\rightarrow 0^+} \kappa_0(\alpha,t) = 0, & \forall\;t\in\mathcal{I}, \nonumber \\ 
 \lim_{\alpha\rightarrow 1^-} \kappa_1(\alpha,t) = 0, & \displaystyle\lim_{\alpha\rightarrow 1^-} \kappa_0(\alpha,t) = 1, & \forall\;t\in\mathcal{I}, \label{kappaconditions} \\
 \kappa_1(\alpha,t)\ne 0, \alpha\in[0,1), & \kappa_0(\alpha,t)\ne 0, \alpha\in(0,1], & \forall\;t\in\mathcal{I}. \nonumber
\end{eqnarray}
Then the following differential operator $D^{\alpha}$, defined via
\begin{equation}\label{derivdef}
 D^{\alpha} f(t) = \kappa_1(\alpha,t) f(t) + \kappa_0(\alpha,t) f'(t), \qquad t\in\mathcal{I}
\end{equation} 
is conformable provided the right-hand side exists at $t$, where $f':=\frac{d}{dt}f$.
\end{definition}

\begin{remark} For the operator given in \eqref{derivdef}, $\kappa_1$ is a type of proportional gain $\kappa_p$, $\kappa_0$ is a type of derivative gain $\kappa_d$, $f$ is the error, and $u=D^{\alpha}f$ is the controller output. 
To illustrate, one could take $\kappa_1\equiv \cos\left(\alpha\pi/2\right)$ and $\kappa_0\equiv \sin\left(\alpha\pi/2\right)$, or
$\kappa_1\equiv (1-\alpha)\omega^{\alpha}$ and $\kappa_0\equiv \alpha\omega^{1-\alpha}$ for any $\omega\in(0,\infty)$; or, $\kappa_1=(1-\alpha)|t|^{\alpha}$  and $\kappa_0=\alpha |t|^{1-\alpha}$ on $\mathcal{I}=\R\backslash\{0\}$, so that
\[ D^{\alpha} f(t) = (1-\alpha) |t|^{\alpha} f(t) + \alpha |t|^{1-\alpha} f'(t). \]
If $\kappa_1$ and $\kappa_0$ are constant then $D^\beta D^\alpha = D^\alpha D^\beta$, but $D^\beta D^\alpha \ne D^\alpha D^\beta$ for $\alpha,\beta\in[0,1]$ in general. 
\end{remark}

\begin{remark}
Note that $D^{\alpha} f(t)$ may exist even if $f'(t)$ does not if we relax the conditions
\[ \kappa_1(\alpha,t)\ne 0, \;\alpha\in[0,1), \quad \kappa_0(\alpha,t)\ne 0, \;\alpha\in(0,1], \quad \forall\;t\in\mathcal{I}. \]
For example, take $\kappa_1=(1-\alpha)t^{\alpha}$ and $\kappa_0=\alpha t^{1-\alpha}$ on $\mathcal{I}=[0,\infty)$, and let $f(t)=t^r$ for $r\in[\alpha,1)$. Then
\[ D^{\alpha} f(t) = (1-\alpha)t^{\alpha+r} + r\alpha t^{r-\alpha} \]
exists for all $t\in[0,\infty)$, even though $f'(t)$ exists only for $t\in(0,\infty)$. We will not explore such singular cases in this paper; they remain open questions for further research. $\hfill\Box$
\end{remark}

In this present work we aim to extend these notions using \eqref{derivdef} to the conformable second-order differential equation
\begin{equation}\label{introeq}
 D^\alpha\left[pD^\alpha x\right](t)+q(t)x(t)=0,
\end{equation}
which will be shown to be associated with a self-adjoint operator because it admits a Lagrange identity. The self-adjoint form \eqref{introeq}, however, is an appropriate generalization and extension of the classical second-order self-adjoint form from ordinary differential equations \cite{cod,kp,r56,Reid,Reid2,Reid3}
\begin{equation}\label{oldform} 
 (px')'(t)+q(t)x(t)=0 
\end{equation}
to conformable equations, since at $\alpha=1$ we see that \eqref{introeq} reduces to \eqref{oldform}.

The paper is constructed as follows, with \eqref{derivdef} as our choice for $D^{\alpha}$, and without loss of generality using the notation $\kappa_j$ as a shorthand for $\kappa_j(\alpha,\cdot)$ satisfying \eqref{kappaconditions} for $j=0,1$. 
Section~\ref{nec} illustrates the interesting geometry associated with this calculus, and covers essential results for subsequent development.
In Section~\ref{secselfadj} we explore \eqref{introeq}, show how it is formally a self-adjoint equation, use a Wronskian and establish a Lagrange identity and Abel's formula. Section~\ref{secreduct} contains several reduction of order theorems. A corresponding Riccati differential equation is introduced in Section~\ref{riccatieq}, the solutions of which are related back to solutions of \eqref{introeq}. Section~\ref{secquad} explores a related quadratic functional and a foundational result, a Picone identity. A full connection among the various aspects of the paper is established in Section~\ref{secreid} in a Reid roundabout theorem.

%%%%%%%%%%%%%%%%%%%%%%%%%%%%%%%%%%%%%%%%
%  Section 2                           %
%  Proportional  Calculus %
%%%%%%%%%%%%%%%%%%%%%%%%%%%%%%%%%%%%%%%%

\section{Calculus for the Proportional Derivative}\label{nec}

The calculus results in this section are modeled after those found in \cite{gracelaura}.

We begin with a vital definition \cite[Definition 1.6]{au}, which establishes a type of exponential function for derivative \eqref{derivdef}.

% Definition: Exponential %

\begin{definition}[Conformable Exponential Function]
Let $\alpha\in(0,1]$, the points $s,t\in\R$ with $s\le t$, and let the function $p:[s,t]\rightarrow\R$ be continuous. Let $\kappa_0,\kappa_1:[0,1]\times\R\rightarrow[0,\infty)$ be continuous and satisfy \eqref{kappaconditions}, with $p/\kappa_0$ and $\kappa_1/\kappa_0$ Riemann integrable on $[s,t]$. Then the conformable exponential function with respect to $D^{\alpha}$ in \eqref{derivdef} is defined to be
\begin{equation}\label{epts}
 e_p(t,s):= e^{\int_s^t \frac{p(\tau)-\kappa_1(\alpha,\tau)}{\kappa_0(\alpha,\tau)} d\tau}, \quad 
 e_0(t,s) = e^{-\int_s^t \frac{\kappa_1(\alpha,\tau)}{\kappa_0(\alpha,\tau)} d\tau},
\end{equation} 
and satisfies
\begin{equation}\label{expderiv}
 D^{\alpha} e_p(t,s) = p(t) e_p(t,s), \quad D^{\alpha} e_0(t,s) = 0.
\end{equation}
\end{definition}

The following useful properties of the exponential function are a direct result of \eqref{epts}. 

% Lemma Exponential Properties %

\begin{lemma}[Exponential Function Properties]\label{expprops}
Let $\kappa_0,\kappa_1$ satisfy \eqref{kappaconditions}, $D^{\alpha}$ satisfy \eqref{derivdef}, let $p,q$ be continuous functions, and let $t,s,r\in\R$. For the conformable exponential function given in \eqref{epts}, the following properties hold.
\begin{enumerate}
 \item[$(i)$] $e_p(t,t)\equiv 1$.
 \item[$(ii)$] $e_p(t,s)e_p(s,r)=e_p(t,r)$.
 \item[$(iii)$] $\frac{1}{e_p(t,s)}=e_p(s,t)=e_{(2\kappa_1-p)}(t,s)$.
 \item[$(iv)$] $e_{\kappa_1}(t,s)\equiv 1$.
 \item[$(v)$] $e_{-\kappa_1}(t,s)=e^2_0(t,s)$.
 \item[$(vi)$] $e_p(t,s)e_q(t,s)=e_{(p+q-\kappa_1)}(t,s)$.
 \item[$(vii)$] $\frac{e_p(t,s)}{e_q(t,s)}=e_{(p-q+\kappa_1)}(t,s)$.
 \item[$(viii)$] $D_t^{\alpha}[e_p(s,t)]=(2\kappa_1-p)(t)e_p(s,t)$, where $D_t^{\alpha}$ indicates employing $D^{\alpha}$ with respect to $t$ for fixed $s$.
 \item[$(ix)$] If $p$ is differentiable and positive, then 
      \[ e_{\left(\frac{-\D^{\alpha}p}{p}\right)}(t,t_0) = \frac{p(t_0)}{p(t)}e^2_0(t,t_0) \quad\text{and}\quad 
         e_{\left(\frac{\D^{\alpha}p}{p}\right)}(t,t_0) = \frac{p(t)}{p(t_0)}. \]
\end{enumerate}
\end{lemma}

% Remark %

\begin{remark}
Due to \eqref{expderiv}, the geodesics of this operator, namely those curves with zero acceleration \cite{mccleary,oneill} and thus satisfy the differential equation
\[ D^{\alpha}D^{\alpha}y = 0, \] 
include the exponentials \eqref{epts} of the form $ce_0(t,t_0)$ for all $c\in\R$. Two important geodesics are the secant line for a function $f$ from $a$ to $b$ given by
\begin{equation}\label{secant}
 \sigma(t):=e_0(t,a)f(a)+h_1(t,a)\frac{e_0(t,b)f(b)-e_0(t,a)f(a)}{h_1(b,a)}, \quad h_1(t,a):=\int_a^t 1 d_{\alpha}s,
\end{equation}
and the tangent line for a function $f$ differentiable at $a$ given by
\begin{equation}\label{tangent}
 \ell(t):=e_0(t,a)f(a)+h_1(t,a)e_0(t,a)D^{\alpha}f(a).
\end{equation}
Despite the fact that the geodesics are non-straight lines, the geometry associated with the proportional derivative calculus in this paper can be considered Euclidean; this is so since given any line (geodesic) and a point not on that line, there is exactly one line through the given point that is parallel to the given line. $\hfill\Box$
\end{remark}

% Rolle's Theorem %

\begin{theorem}[Rolle's Theorem]\label{rolle}
Let $\alpha\in(0,1]$. If the function $f$ is continuous on $[a,b]$ and differentiable on $(a,b)$, with
\[ f(a) = e_0(a,b)f(b), \]
then there exists at least one number $c\in(a,b)$ such that $D^{\alpha}f(c)=0$.
\end{theorem}

%\begin{proof}
%Let $f$ satisfy the hypotheses of the theorem, and set
%\[ g(t):= e_0(a,t)f(t) - f(a). \]
%Then $g(a)=0$, and $g(b)=0$ by assumption. Now $g$ is continuous on $[a,b]$ and differentiable on $(a,b)$, so by classical Rolle's Theorem there exists %$c\in(a,b)$ such that $g'(c)=0$. Thus we have
%\begin{eqnarray*}
% 0 &=& g'(c) = e_0(a,c)f'(c) + f(c)e_0(a,c)\frac{\kappa_1(c)}{\kappa_0(c)} \\
% &=& \frac{e_0(a,c)}{\kappa_0(c)}\left( \kappa_0(c) f'(c) + \kappa_1(c) f(c) \right) = \frac{e_0(a,c)}{\kappa_0(c)} D^{\alpha}f(c) 
%\end{eqnarray*}
%by \eqref{derivdef}, which yields $D^{\alpha}f(c)=0$ as $e_0(a,c)$ is never zero.
%\end{proof}

% Mean Value Theorem %

\begin{theorem}[Mean Value Theorem]\label{mvt}
Let $\alpha\in(0,1]$. If the function $f$ is continuous on $[a,b]$ and differentiable on $(a,b)$, then there exists at least one number $c\in(a,b)$ such that 
\[ D^{\alpha}f(c) = \frac{e_0(c,b)f(b) - e_0(c,a)f(a)}{h_1(b,a)}, \]
where $h_1(b,a)=\int_a^b 1 d_{\alpha}s$ for $d_{\alpha}s:=ds/\kappa_0(s)$.
\end{theorem}

%\begin{proof}
%Let $f$ satisfy the hypotheses of the theorem, and set
%\[ g(t):= f(t) - e_0(t,a)f(a) - \frac{h_1(t,a)}{h_1(b,a)}\left(e_0(t,b)f(b)-e_0(t,a)f(a)\right), \]
%where $h_1(t,a)=\int_a^t 1 d_{\alpha}s$.
%Then $g$ is continuous on $[a,b]$, differentiable on $(a,b)$, and $g(a)=0=g(b)$. By Rolle's Theorem \ref{rolle} above, there exists $c\in(a,b)$ such that %$D^{\alpha}g(c)=0$. Thus we have
%\begin{eqnarray*}
% 0 &=& D^{\alpha}g(c) = D^{\alpha}f(c) - 0 - \left(\frac{e_0(c,b)f(b)-e_0(c,a)f(a)}{h_1(b,a)}\right),
%\end{eqnarray*}
%which gives the result.
%\end{proof}

The following fundamental theorem was given in \cite[Lemma 1.9 (ii)]{au}, but we supply a more rigorous proof here as Rolle's Theorem and the Mean Value Theorem above are new.

% Fundamental Theorem of Integral Calculus %

\begin{theorem}[Fundamental Theorem of Integral Calculus]\label{ftc}
Let $\alpha\in(0,1]$. Suppose $f:[a,b]\rightarrow\R$ is differentiable on $[a,b]$ and $f'$ is integrable on $[a,b]$. Then
\[ \int_a^b D^{\alpha}[f(t)]e_0(b,t)d_{\alpha}t = f(b) - f(a)e_0(b,a), \]
where $d_{\alpha}t:=dt/\kappa_0(t)$.
\end{theorem}

\begin{proof}
Let $P$ be any partition of $[a,b]$, $P=\{t_0,t_1,\cdots,t_n\}$, and recall that $h_1(t,a)=\int_a^t 1 d_{\alpha}s$. By the Mean Value Theorem \ref{mvt} above applied to $f$ on $[t_{i-1},t_i]$, there exist $c_i\in(t_{i-1},t_i)$ such that
\[ D^{\alpha}f(c_i) = \frac{e_0(c_i,t_i)f(t_i) - e_0(c_i,t_{i-1})f(t_{i-1})}{h_1(t_i,t_{i-1})}, \]
or equivalently
\[ D^{\alpha}[f(c_i)]e_0(b,c_i)h_1(t_i,t_{i-1}) = e_0(b,t_i)f(t_i) - e_0(b,t_{i-1})f(t_{i-1}). \] 
After forming the Riemann-Stieltjes sum
\[ S(P,f,\mu) = \sum_{i=1}^{n} D^{\alpha}[f(c_i)]e_0(b,c_i)\left(\mu(t_i) - \mu(t_{i-1})\right), \quad \mu(t):=\int_a^t 1 d_{\alpha}s=h_1(t,a), \]
we see that 
\[ S(P,f,\mu) = \sum_{i=1}^{n} e_0(b,t_i)f(t_i) - e_0(b,t_{i-1})f(t_{i-1}) = f(b) - e_0(b,a)f(a). \]
Since $P$ was arbitrary, 
\begin{eqnarray*} 
 \int_a^b D^{\alpha}[f(t)]e_0(b,t) d\mu(t) &=& \int_a^b D^{\alpha}[f(t)]e_0(b,t) \mu'(t) dt \\
 &=& \int_a^b D^{\alpha}[f(t)]e_0(b,t) d_{\alpha}t \\
 &=& f(b) - e_0(b,a)f(a).
\end{eqnarray*}
This completes the proof.
\end{proof}

% Example %

%\begin{example}
%Let $\alpha\in(0,1]$ and $r>0$. In this example we evaluate
%\[ \int_a^b h^{r-1}_1(t,a)[r+\kappa_1(t)h_1(t,a)]e_0(b,t)d_{\alpha}t, \qquad h_1(t,a):=\int_a^t \frac{1}{\kappa_0(\tau)}d\tau = \int_a^t 1d_{\alpha}\tau \]
%using the fundamental theorem of integral calculus, Theorem \ref{ftc}. Note that 
%\begin{eqnarray*} 
% h^{r-1}_1(t,a)[r+\kappa_1(t)h_1(t,a)] &=& rh^{r-1}_1(t,a) + \kappa_1(t)h^{r}_1(t,a) \\
% &=& \kappa_0(t)rh^{r-1}_1(t,a)\frac{1}{\kappa_0(t)} + \kappa_1(t)h^{r}_1(t,a) \\
% &=& D^{\alpha}\left[h^r_1(t,a)\right].
%\end{eqnarray*}
%Thus by Theorem \ref{ftc},
%\begin{eqnarray*} 
% \int_a^b h^{r-1}_1(t,a)[r+\kappa_1(t)h_1(t,a)]e_0(b,t)d_{\alpha}t 
% &=& \int_a^b D^{\alpha}\left[h^r_1(t,a)\right]e_0(b,t)d_{\alpha}t = h^r_1(b,a).
%\end{eqnarray*}
%This concludes the example. $\hfill\triangle$
%\end{example}

For reference, we include the following results \cite[Lemma 1.7]{au}.

\begin{lemma}[Basic Derivatives]\label{basicderiv}
Let the conformable differential operator $D^{\alpha}$ be given as in \eqref{derivdef}, where $\alpha\in[0,1]$. Let the function $p:[s,t]\rightarrow\R$ be continuous. Let $\kappa_0,\kappa_1:[0,1]\times\R\rightarrow[0,\infty)$ be continuous and satisfy \eqref{kappaconditions}, with $p/\kappa_0$ and $\kappa_1/\kappa_0$ Riemann integrable on $[s,t]$. Assume the functions $f$ and $g$ are differentiable as needed. Then
\begin{enumerate}
\item $D^{\alpha}[af+bg] = aD^{\alpha}[f]+bD^{\alpha}[g]$ for all $a,b\in\R$;
\item $D^{\alpha}c = c\kappa_1(t)$ for all constants $c\in\R$;
\item $D^{\alpha}[fg] = fD^{\alpha}[g] + gD^{\alpha}[f] - fg\kappa_1$;
\item $D^{\alpha}[f/g] = \frac{gD^{\alpha}[f]-fD^{\alpha}[g]}{g^2} + \frac{f}{g}\kappa_1$; 
\item for $\alpha\in(0,1]$ and for the exponential function $e_0$ given in \eqref{epts}, we have
      \begin{equation}\label{DofI}
        D^{\alpha}\left[\int_{a}^t f(s)e_0(t,s)d_{\alpha}s\right] = f(t), \quad d_{\alpha}s=\kappa_0^{-1}(s)ds.
      \end{equation}
\end{enumerate}
\end{lemma}

%%%%%%%%%%%%%%%%%%%%%%%%%%%%%%%%%%%%%%%%%%%%%%%%%%%%%%%%%%%%%%%%%%%%%%%%%
% Definition of Maxima and Minima
%%%%%%%%%%%%%%%%%%%%%%%%%%%%%%%%%%%%%%%%%%%%%%%%%%%%%%%%%%%%%%%%%%%%%%%%%

\begin{definition}\label{maxmin}
A function $f$ has a maximum at $t_0\in[a,b]$ if and only if
\[ f(t_0) \ge e_0(t_0,t)f(t) \quad \text{for all}\quad t\in[a,b]. \]
Similarly, a function $f$ has a minimum at $t_0\in[a,b]$ if and only if
\[ f(t_0) \le e_0(t_0,t)f(t) \quad \text{for all}\quad t\in[a,b]. \]
\end{definition}

% Definition: Increasing/Decreasing %

\begin{definition}
A function $f$ is $\alpha$-increasing on an interval $\mathcal{I}$ if
\[ e_0(t_1,t_2)f(t_2) \ge f(t_1), \quad\text{whenever}\quad t_2 > t_1, \quad t_1,t_2\in\mathcal{I}, \]
and is strictly $\alpha$-increasing if
\[ e_0(t_1,t_2)f(t_2) > f(t_1), \quad\text{whenever}\quad t_2 > t_1, \quad t_1,t_2\in\mathcal{I}. \]
A function $f$ is $\alpha$-decreasing on an interval $\mathcal{I}$ if
\[ f(t_2) \le e_0(t_2,t_1)f(t_1), \quad\text{whenever}\quad t_2 > t_1, \quad t_1,t_2\in\mathcal{I}, \]
and is strictly $\alpha$-decreasing if
\[ f(t_2) < e_0(t_2,t_1)f(t_1), \quad\text{whenever}\quad t_2 > t_1, \quad t_1,t_2\in\mathcal{I}. \]
\end{definition}

% Increasing/Decreasing Test %

\begin{theorem}[Increasing/Decreasing Test]\label{incdectest}
Suppose that $D^{\alpha}f$ exists on some interval $\mathcal{I}$.
\begin{enumerate}
 \item If $D^{\alpha}f(t)>0$ for all $t\in\mathcal{I}$, then $f$ is strictly $\alpha$-increasing on $\mathcal{I}$. 
 \item If $D^{\alpha}f(t)<0$ for all $t\in\mathcal{I}$, then $f$ is strictly $\alpha$-decreasing on $\mathcal{I}$. 
\end{enumerate} 
\end{theorem}

%\begin{proof}
%Let $t_1,t_2\in\mathcal{I}$ with $t_2>t_1$. Since $D^{\alpha}f(t)>0$ for all $t\in\mathcal{I}$, $f$ is differentiable on $[t_1,t_2]$, and the Mean Value Theorem %\ref{mvt} yields a number $c\in(t_1,t_2)$ such that
%\[ \frac{e_0(c,t_2)f(t_2) - e_0(c,t_1)f(t_1)}{h_1(t_2,t_1)} = D^{\alpha}f(c) > 0. \]
%Thus $e_0(c,t_2)f(t_2) > e_0(c,t_1)f(t_1)$, or equivalently $e_0(t_1,t_2)f(t_2) > f(t_1)$ after multiplying both sides by $e_0(t_1,c)>0$ and using Lemma %\ref{expprops}. This proves that $f$ is strictly $\alpha$-increasing on $\mathcal{I}$. The proof of the case where $D^{\alpha}f(t)<0$ for all $t\in\mathcal{I}$ %is similar and hence omitted.
%\end{proof}

% Concavity Test %

\begin{theorem}[Concavity Test]
Suppose that $D^{\alpha}D^{\alpha}f$ exists on some interval $\mathcal{I}$. 
\begin{enumerate}
 \item If $D^{\alpha}D^{\alpha}f(t)>0$ for all $t\in\mathcal{I}$, then the graph of $f$ is concave upward on $\mathcal{I}$. 
 \item If $D^{\alpha}D^{\alpha}f(t)<0$ for all $t\in\mathcal{I}$, then the graph of $f$ is concave downward on $\mathcal{I}$. 
\end{enumerate}
Here concave upward means the curve $y=f(t)$ lies above all of its tangent lines \eqref{tangent} on $\mathcal{I}$, and concave downward means the curve $y=f(t)$ lies below all of its tangent lines \eqref{tangent} on $\mathcal{I}$.
\end{theorem}

%\begin{proof}
%Let $a,t\in\mathcal{I}$. We will show that the curve $y=f(t)$ lies above the tangent line \eqref{tangent} through the point $(a,f(a))$. First, let $t>a$. By the %Mean Value Theorem \ref{mvt} applied to $f$ on $[a,t]$, there exists a number $c\in(a,t)$ such that 
%\[ D^{\alpha}f(c) = \frac{e_0(c,t)f(t) - e_0(c,a)f(a)}{h_1(t,a)}; \]
%solving this for $f(t)$ we have
%\[ f(t) = e_0(t,a)f(a) + h_1(t,a)e_0(t,c)D^{\alpha}f(c). \]
%Since $D^{\alpha}[D^{\alpha}f](t)>0$ for all $t\in\mathcal{I}$, we know that $D^{\alpha}f$ is strictly $\alpha$-increasing on $\mathcal{I}$, and thus 
%\[ e_0(a,t)D^{\alpha}f(t) > D^{\alpha}f(a), \quad t\in\mathcal{I}, \quad t>a, \]
%so that in particular we have
%\[ e_0(t,c)D^{\alpha}f(c) > e_0(t,a)D^{\alpha}f(a). \] 
%Consequently, 
%\begin{eqnarray*}
% f(t) &=& e_0(t,a)f(a) + h_1(t,a)e_0(t,c)D^{\alpha}f(c) > e_0(t,a)f(a) + h_1(t,a)e_0(t,a)D^{\alpha}f(a) = \ell(t),
%\end{eqnarray*}
%where $\ell$ is the tangent line to $f$ at $a$ given in \eqref{tangent}. If $t<a$, then 
%\[ e_0(t,c)D^{\alpha}f(c) < e_0(t,a)D^{\alpha}f(a), \] 
%but multiplication by the negative number $h_1(t,a)$ reverses the inequality, so the proof is still valid. The case where $D^{\alpha}D^{\alpha}f(t)<0$ for all %$t\in\mathcal{I}$ is similar and thus omitted.
%\end{proof}

% Example %

\begin{example}
In this example we illustrate the interesting geometry associated with this derivative. Let $\alpha\in(0,1)$, $\kappa_1=\cos(\alpha\pi/2)$, $\kappa_0=\sin(\alpha\pi/2)$, and let $f(t)=\sin t$. Then
\[ D^{\alpha}[f(t)] = \cos(\alpha\pi/2)\sin t + \sin(\alpha\pi/2)\cos t = \sin\left(t+\frac{\alpha\pi}{2}\right). \]
Setting this equal to zero to determine the critical numbers, we find
\[ t_n = n\pi-\frac{\alpha\pi}{2}, \quad n\in\Z. \]
Note that as $\alpha\rightarrow 0^+$ we have $t_n=n\pi$, the zeros of the original function $\sin t$; as $\alpha\rightarrow 1^-$ we have $t_n=n\pi-\frac{\pi}{2}$, the zeros of the full derivative function $\cos t$, as expected.
After noticing that 
\[ D^{\alpha}[D^{\alpha}f](t)=\sin\left(t+\alpha\pi\right), \]
we see that
\[ D^{\alpha}[D^{\alpha}f](t_n) = \sin\left(n\pi+\frac{\alpha\pi}{2}\right)=(-1)^n\sin\left(\frac{\alpha\pi}{2}\right)  \begin{cases} <0 &: n \text{ odd}, \\  >0 &: n \text{ even}. \end{cases} \]
Thus $f(t)=\sin t$ has an $\alpha$-max at $t_n$ when $n$ is odd, and an $\alpha$-min at $t_n$ when $n$ is even. Since
\[ \sin(t_n) = \sin\left(n\pi-\frac{\alpha\pi}{2}\right) = (-1)^{n+1}\sin\left(\frac{\alpha\pi}{2}\right), \]
we conclude that $f(t)=\sin t$ has an $\alpha$-max of $\sin\left(\frac{\alpha\pi}{2}\right)$ at $(2m+1)\pi-\frac{\alpha\pi}{2}$ for $m\in\Z$, and an $\alpha$-min of $-\sin\left(\frac{\alpha\pi}{2}\right)$ at $2m\pi-\frac{\alpha\pi}{2}$ for $m\in\Z$.
\end{example}

%\begin{example}
%In this example we illustrate the interesting geometry associated with this derivative. Let $\alpha\in(0,1)$, $\kappa_1=(1-\alpha)$, $\kappa_0=\alpha$, and let %$f(t)=\sin t$. Then
%\[ D^{\alpha}[f(t)] = \alpha \cos t + (1-\alpha) \sin t. \]
%Setting this equal to zero to determine the critical numbers, we find
%\[ t_n = n\pi - \tan^{-1}\left(\frac{\alpha}{1-\alpha}\right), \quad n\in\Z. \]
%Note that as $\alpha\rightarrow 0^+$ we have $t_n=n\pi$, the zeros of the original function $\sin t$; as $\alpha\rightarrow 1^-$ we have %$t_n=-\frac{\pi}{2}+n\pi$, the zeros of the full derivative function $\cos t$, as expected.
%After calculating $D^{\alpha}[D^{\alpha}f(t)]$, we see that
%\[ D^{\alpha}[D^{\alpha}f(t_n)]  \begin{cases} <0 &: n \text{ odd}, \\  >0 &: n \text{ even}. \end{cases} \]
%Thus $f(t)=\sin t$ has an $\alpha$-max at $t_n$ when $n$ is odd, and an $\alpha$-min at $t_n$ when $n$ is even. Since
%\[ \sin(t_n) = \frac{(-1)^{n+1}\alpha}{\sqrt{2\alpha^2-2\alpha+1}}, \]
%we conclude that $f(t)=\sin t$ has an $\alpha$-max of $\frac{\alpha}{\sqrt{2\alpha^2-2\alpha+1}}$ at $n\pi-\tan^{-1}\left(\frac{\alpha}{1-\alpha}\right)$ when %$n$ is odd, and an $\alpha$-min of $\frac{-\alpha}{\sqrt{2\alpha^2-2\alpha+1}}$ at $n\pi-\tan^{-1}\left(\frac{\alpha}{1-\alpha}\right)$ when $n$ is even.
%\end{example}

%%%%%%%%%%%%%%%%%%%%%%%%%%%%%%%%%%%%%%%%
%  Section 3                           %
%  Self adjoint differential equations %
%%%%%%%%%%%%%%%%%%%%%%%%%%%%%%%%%%%%%%%%

\section{Self-Adjoint Differential Equations}\label{secselfadj}
Let $p$ and $q$ be continuous functions on some interval $\mathcal{I}\subseteq\R$ with $p(t)>0$ for all $t\in\mathcal{I}$. In this section we are concerned with the conformable second-order (formally) self-adjoint homogeneous differential equation $Lx=0$ or for a continuous function $h$ the associated nonhomogeneous equation 
\begin{equation}\label{saeq}
 Lx=h, \quad Lx(t):= D^\alpha\left[pD^\alpha x\right](t) + q(t)x(t)
\end{equation}
for $t\in\mathcal{I}$, $\alpha\in(0,1]$. We interpret $D^{\alpha}$ using \eqref{derivdef} for $\kappa_0$ and $\kappa_1$ satisfying \eqref{kappaconditions}. This operator $L$ is simpler and more useful than the one introduced in \cite[(4.1)]{au}.

%%%%%%%%%%%%%%
% Definition %
%%%%%%%%%%%%%%

\begin{definition} 
Let $\D$ denote the set of all continuous functions $x$ defined on $\mathcal{I}$ such that $D^{\alpha}x$ and $D^\alpha\left[pD^\alpha x\right]$ are continuous on $\mathcal{I}$. Then $x$ is a solution of the homogeneous equation $Lx=0$ for $L$ in \eqref{saeq} on $\mathcal{I}$ provided $x\in\D$ and
\[ Lx(t) = 0 \quad\text{for all}\quad t\in\mathcal{I}. \]
\end{definition}

Next we address the question under which circumstances an equation of the form 
\[ a(t)D^{\alpha}D^{\alpha}x(t) + b(t)D^{\alpha}x(t) + c(t)x(t) = g(t) \]
can be rewritten in self-adjoint form \eqref{saeq}, where we assume that $a\neq 0$ on $\mathcal{I}$ and $a,b,c,g$ are continuous functions on $\mathcal{I}$.

% Theorem %

\begin{theorem}\label{pisaf} 
Assume $\kappa_0,\kappa_1$ satisfy \eqref{kappaconditions}, and $t_0\in\mathcal{I}$. If $a,b,c,g:\mathcal{I}\rightarrow\R$ are continuous functions with $a\neq 0$ on $\mathcal{I}$, then the iterated conformable equation
\begin{equation}\label{lineq} 
  a(t)D^{\alpha}D^{\alpha}x + b(t)D^{\alpha}x + c(t)x = g(t), \quad t\in\mathcal{I},
\end{equation}
can be written in self-adjoint form $Lx=h$, where $h=gp/a$ for
\begin{equation}\label{peq} 
  p(t) = e_{(\frac{b}{a}+\kappa_1)}(t,t_0), \quad q(t) = p(t)c(t)/a(t), 
\end{equation}
for $t\in\mathcal{I}$.
\end{theorem}

\begin{proof} 
Assume $x$ is a solution of \eqref{lineq}. By \eqref{peq}, after suppressing the arguments, we see that
\begin{equation}\label{papb}
 pb/a = D^{\alpha}p-\kappa_1 p.
\end{equation}
Multiplying both sides of \eqref{lineq} by $p/a$ and using \eqref{papb}, we get 
\begin{eqnarray*}
 h &=& \frac{gp}{a} = pD^{\alpha}D^{\alpha}x + \frac{pb}{a}D^{\alpha}x + \frac{pc}{a}x \\
 &=& pD^{\alpha}[D^{\alpha}x] + (D^{\alpha}p-\kappa_1p)D^{\alpha}x + qx \\
 &=& D^{\alpha}[pD^{\alpha}x] + qx = Lx.
\end{eqnarray*}
The proof is complete.
\end{proof}

% Example %

\begin{example}\label{coscosh}
Let $L$ be as in \eqref{saeq}, and let $\kappa_0,\kappa_1$ satisfy \eqref{kappaconditions}. In \eqref{lineq} choose the constant coefficients $a\equiv 1$, $b\equiv 0$, and $c\equiv \pm \omega^2$ for constant $\omega\in\R$ with $\omega>0$, and $g\equiv 0$. By Lemma \ref{expprops} and Theorem \ref{pisaf} we have
\[ p(t)\equiv 1, \quad q(t)= \pm \omega^2, \]
and the homogeneous self-adjoint equation $Lx=0$ takes the form
\[ D^\alpha  D^\alpha x(t) \pm \omega^2 x(t) = 0. \]
Using \cite[Theorem 3.1]{au}, this has characteristic equation $\lambda^2 \pm \omega^2 = 0$ and solutions
\begin{eqnarray*} 
 x_+(t) &=& c_1e_0(t,t_0)\cos\left(\int_{t_0}^t \omega d_{\alpha}s \right) + c_2e_0(t,t_0)\sin\left(\int_{t_0}^t \omega d_{\alpha}s \right), \\
 x_-(t) &=& c_1e_{\omega}(t,t_0) + c_2e_{-\omega}(t,t_0).
\end{eqnarray*}
These solutions suggest the notation
\begin{eqnarray*}
 \cos_{\alpha}(\omega;t,t_0) &=& e_0(t,t_0)\cos\left(\int_{t_0}^t \omega d_{\alpha}s \right), \quad \sin_{\alpha}(\omega;t,t_0) = e_0(t,t_0)\sin\left(\int_{t_0}^t \omega d_{\alpha}s \right), \\
 \cosh_{\alpha}(\omega;t,t_0) &=& \frac{e_{\omega}(t,t_0)+e_{-\omega}(t,t_0)}{2}, \quad \sinh_{\alpha}(\omega;t,t_0) = \frac{e_{\omega}(t,t_0)-e_{-\omega}(t,t_0)}{2},
\end{eqnarray*}
and the equations
\[ D^{\alpha}\cos_{\alpha}(\omega;t,t_0) = -\omega\sin_{\alpha}(\omega;t,t_0), \quad D^{\alpha}\sin_{\alpha}(\omega;t,t_0) = \omega\cos_{\alpha}(\omega;t,t_0), \]
\[ D^{\alpha}\cosh_{\alpha}(\omega;t,t_0) = \omega\sinh_{\alpha}(\omega;t,t_0), \quad D^{\alpha}\sinh_{\alpha}(\omega;t,t_0) = \omega\cosh_{\alpha}(\omega;t,t_0) \]
as well as
\[ \cos^2_{\alpha}(\omega;t,t_0)+\sin^2_{\alpha}(\omega;t,t_0) = e^2_0(t,t_0) = \cosh^2_{\alpha}(\omega;t,t_0)-\sinh^2_{\alpha}(\omega;t,t_0) \]
all hold. In summary,
\[ D^\alpha  D^\alpha x(t) + \omega^2 x(t) = 0 \]
has general solution
\[ x(t) = c_1\cos_{\alpha}(\omega;t,t_0) + c_2\sin_{\alpha}(\omega;t,t_0), \]
and
\[ D^\alpha  D^\alpha x(t) - \omega^2 x(t) = 0 \]
has general solution
\[ x(t) = c_1\cosh_{\alpha}(\omega;t,t_0) + c_2\sinh_{\alpha}(\omega;t,t_0). \]
This ends the example. $\hfill\triangle$
\end{example}

We next state a theorem concerning the existence-uniqueness of solutions of initial value problems for the nonhomogeneous self-adjoint equation $Lx=h$. 

%%%%%%%%%%%%%%%%%%%%%%%%%%%%%%
% Thm (Existence-Uniqueness) %
%%%%%%%%%%%%%%%%%%%%%%%%%%%%%%

\begin{theorem}\label{unique} 
	\index{existence uniqueness theorem!self-adjoint problem}
Assume $\kappa_0,\kappa_1$ satisfy \eqref{kappaconditions}. Let $\alpha\in(0,1]$, $t_0\in\mathcal{I}$, and let $D^{\alpha}$ be as in \eqref{derivdef}.
Assume $p,q,h$ are continuous on $\mathcal{I}$ with $p(t)\ne 0$, and suppose $x_0,x_1\in\R$ are given constants. Then the initial value problem 
\[ Lx = h(t), \quad x(t_0)=x_0, \quad D^{\alpha}x(t_0)=x_1 \] 
has a unique solution that exists on all of $\mathcal{I}$.
\end{theorem}

\begin{proof}
We will write $Lx=h$ as an equivalent vector equation, and then invoke the standard ($\alpha=1$) result to complete the argument. Let $x\in\D$ such that $x$ is a solution of $Lx=h$, and set 
\[ y:=pD^{\alpha}x, \quad D^{\alpha}x = y/p. \]
Using the fact that $x$ is a solution of $Lx=h$ for $L$ defined in \eqref{saeq}, we have
\[ D^{\alpha}y = -qx+h. \]
Therefore, if we set the vector 
\[ z:= \begin{bmatrix} x \\ y \end{bmatrix}, \]
then $z$ is a solution of the vector equation
\[ D^{\alpha}z = Az+b, \quad A = \begin{bmatrix} 0 & 1/p \\ -q & 0 \end{bmatrix}, \quad b = \begin{bmatrix} 0 \\ h \end{bmatrix}. \]
By the definition of $D^{\alpha}$ in \eqref{derivdef}, we thus have
\[ z' = \begin{bmatrix} \frac{-\kappa_1}{\kappa_0} & \frac{1}{\kappa_0p} \\ \frac{-q}{\kappa_0} & \frac{-\kappa_1}{\kappa_0} \end{bmatrix}z + \begin{bmatrix} 0 \\ \frac{h}{\kappa_0} \end{bmatrix}, \]
and the result follows from the standard ($\alpha=1$) proof, since all the functions here are continuous; see, for example, \cite[Theorem 5.4]{kp}.
\end{proof}

Recall the following definition, \cite[Definition 4.5]{au}.

%%%%%%%%%%%%%%%%%%%%%%%%
% Definition Wronskian %
%%%%%%%%%%%%%%%%%%%%%%%%

\begin{definition}[Wronskian]\label{wdef}
Let $\kappa_0,\kappa_1:[0,1]\times\mathcal{I}\rightarrow[0,\infty)$ be continuous and satisfy \eqref{kappaconditions}. If $x,y:\mathcal{I}\to\R$ are differentiable on $\mathcal{I}$, then the conformable Wronskian of $x$ and $y$ is given by 
\begin{equation}\label{Wronskian}
 W(x,y)(t)=\det\mat{x(t)}{y(t)}{D^{\alpha}x(t)}{D^{\alpha}y(t)} \quad \text{for} \quad t\in\mathcal{I},
\end{equation}
for $D^{\alpha}$ given in \eqref{derivdef}.
\end{definition}

%%%%%%%%%%%%%%%%%%%%%%%%%%%%%
% Theorem Lagrange Identity %
%%%%%%%%%%%%%%%%%%%%%%%%%%%%%

\begin{theorem}[Conformable Lagrange Identity]\label{wnab} 
Let $L$ be given as in \eqref{saeq}. If $x,y\in\D$, then
\[ x(Ly) - y(Lx) = D^{\alpha}[pW(x,y)] + \kappa_1 pW(x,y), \quad t\in\mathcal{I}, \]
for $L$ given in \eqref{saeq}. Equivalently, we have 
\begin{equation}\label{lagrangeformula}
 e_0(t,b)D^{\alpha}\left[\frac{pW(x,y)(t)}{e_0(t,b)}\right] = x(Ly)-y\left(Lx\right).
\end{equation}
\end{theorem}

\begin{proof}
For $x,y\in\D$, using the product rule from Lemma \ref{basicderiv} we have (suppressing all arguments)
\begin{eqnarray*}
 D^{\alpha}[pW(x,y)]
 &=& D^{\alpha}\left[xpD^{\alpha}y - ypD^{\alpha}x \right] \\
 &=& xD^{\alpha}\left[pD^{\alpha}y\right] + p(D^{\alpha}x)(D^{\alpha}y)-\kappa_1pxD^{\alpha}y \\ 
 & & - yD^{\alpha}\left[pD^{\alpha}x\right] - p(D^{\alpha}x)(D^{\alpha}y)+\kappa_1pyD^{\alpha}x \\
 &=& x(Ly)-y(Lx) -\kappa_1 pW(x,y)
\end{eqnarray*}
on $\mathcal{I}$. Additionally, by the quotient rule we see that 
\begin{equation}\label{dfpluskf}
 e_0(t,b)D^{\alpha}\left[\frac{f(t)}{e_0(t,b)}\right] = D^{\alpha}f + \kappa_1f. 
\end{equation}
for any differentiable function $f$, and thus \eqref{lagrangeformula} holds as well.
\end{proof}

%%%%%%%%%%%%%%%%%%%%%%%%%%%%%
% Definition: Inner Product %
%%%%%%%%%%%%%%%%%%%%%%%%%%%%%

\begin{definition}[Inner Product]
Let $\alpha\in(0,1]$ and $e_0$ be as given in \eqref{epts}. Define the inner product of (continuous) functions $f,g\in C(\mathcal{I})$ on $[a,b]\subseteq\mathcal{I}$ to be
\begin{equation}\label{ip}
 \langle f,g \rangle = \int_a^b f(t)g(t)e^2_0(b,t) d_{\alpha}t, \quad d_{\alpha}t:=\frac{dt}{\kappa_0(t)}.
\end{equation}
\end{definition}

%%%%%%%%%%%%%%%%%%%%%%%%%%%%%
% Corollary Green's Formula %
%%%%%%%%%%%%%%%%%%%%%%%%%%%%%

\begin{corollary}[Green's Formula; Self-Adjoint Operator]\label{selfadj}
Let $L$ be given as in \eqref{saeq}. If $x,y\in\D$, then Green's formula
\begin{equation}\label{greenformula}
 \langle x, Ly \rangle - \langle Ly, x \rangle = e^2_0(b,t)p(t)W(x,y)(t)\Big|_{t=a}^b
\end{equation}
holds. Moreover, the operator $L$ is formally self-adjoint with respect to the inner product \eqref{ip}; that is, the identity 
\[ \langle Lx, y\rangle = \langle x, Ly \rangle \]
holds if and only if $x,y\in\D$ and $x,y$ satisfy the self-adjoint boundary conditions
\begin{equation}\label{sabc}
 e^2_0(b,t)p(t)W(x,y)(t)\Big|_{t=a}^b=0,
\end{equation}
where we have used the conformable Wronskian matrix from Definition $\ref{wdef}$.
\end{corollary}

\begin{proof}
From Theorem \ref{wnab} we have the Lagrange identity \eqref{lagrangeformula} given by
\[ e_0(t,b)D^{\alpha}\left[\frac{pW(x,y)(t)}{e_0(t,b)}\right] = x(Ly)-y\left(Lx\right). \]
If we multiply both sides of this by $e^2_0(b,t)$ and integrate from $a$ to $b$ we obtain
\[ \int_a^b D^{\alpha}\left[\frac{pW(x,y)(t)}{e_0(t,b)}\right]e_0(b,t)d_{\alpha}t = \langle x, Ly \rangle - \langle Ly, x \rangle. \]
By Theorem \ref{ftc} we have 
\[ \int_a^b D^{\alpha}\left[\frac{pW(x,y)(t)}{e_0(t,b)}\right]e_0(b,t)d_{\alpha}t = e_0(b,t)\frac{pW(x,y)(t)}{e_0(t,b)}\big|_{t=a}^b, \]
so that Green's formula \eqref{greenformula} holds. Thus $\langle x, Ly \rangle = \langle Ly, x \rangle$ if and only if $x,y\in\D$ satisfy the self-adjoint boundary conditions \eqref{sabc}. 
\end{proof}

%%%%%%%%%%%%%%%%%%%%%%%%%%%%
% Corollary ABEL'S FORMULA %
%%%%%%%%%%%%%%%%%%%%%%%%%%%%

\begin{corollary}[Abel's Formula]\label{wronskid}
Let $L$ be given as in \eqref{saeq}. If $x,y$ are solutions of $Lx=0$ on $\mathcal{I}$, then for fixed $b\in\mathcal{I}$ we have
\[ W(x,y)(t) = \frac{ce^2_0(t,b)}{p(t)}, \quad t\in\mathcal{I}, \]
where $c=p(b)W(x,y)(b)$ is a constant.
\end{corollary}

\begin{proof}
As in \eqref{lagrangeformula} in the proof of Corollary \ref{selfadj}, for $x,y\in\D$ we have
\[ e_0(t,b)D^{\alpha}\left[\frac{pW(x,y)(t)}{e_0(t,b)}\right] = x(Ly)-y\left(Lx\right). \]
If $x,y$ are solutions of \eqref{saeq} on $\mathcal{I}$, then $Lx=0=Ly$ and 
\[ e_0(t,b)D^{\alpha}\left[\frac{pW(x,y)(t)}{e_0(t,b)}\right] = 0. \]
As a result, 
\[ D^{\alpha}\left[\frac{pW(x,y)(t)}{e_0(t,b)}\right] = 0, \]
so that 
\[ \frac{p(t)W(x,y)(t)}{e_0(t,b)} = ce_0(t,b), \]
where $c$ is the constant $c=p(b)W(x,y)(b)$.
\end{proof}

%%%%%%%%%%%%%%
% Corollary  %
%%%%%%%%%%%%%%

\begin{corollary}
If $x,y$ are solutions of \eqref{saeq} on $\mathcal{I}$, then either $W(x,y)(t)=0$ for all $t\in\mathcal{I}$, or $W(x,y)(t)\ne 0$ for all $t\in\mathcal{I}$.
\end{corollary}

%%%%%%%%%%%
% Theorem %
%%%%%%%%%%%
\begin{theorem}
Equation \eqref{saeq} on $\mathcal{I}$ has two linearly independent solutions, and every solution of \eqref{saeq} on $\mathcal{I}$ is a linear combination of these two solutions.
\end{theorem}

\begin{proof}
The proof is similar to the $\alpha=1$ case; see \cite[Theorem 5.11]{kp}.
\end{proof}

%%%%%%%%%%%%%%%%%%%%%%%%%%%
% Theorem ABEL'S Converse %
%%%%%%%%%%%%%%%%%%%%%%%%%%%

\begin{theorem}[Converse of Abel's Formula]\label{abelcon}
Let $L$ be as in \eqref{saeq}, and let $x$ be a solution of $Lx=0$ on $\mathcal{I}$ such that $x\neq 0$ on $\mathcal{I}$. If $y\in\D$ satisfies $W(x,y)(t) = \frac{ce^2_0(t,b)}{p(t)}$ for some constant $c\in\R$, then $y$ is also a solution of $Lx=0$.
\end{theorem}

\begin{proof}
Suppose that $x$ is a solution of $Lx=0$ such that $x\neq 0$ on $\mathcal{I}$, and assume $y\in\D$ satisfies $W(x,y)(t) = ce^2_0(t,b)/p(t)$ for some constant $c\in\R$. By the Lagrange identity (Theorem \ref{wnab}) and Lagrange's formula \eqref{lagrangeformula} we have for $t\in\mathcal{I}$ that
\[ x(Ly)(t)-y(Lx)(t) = e_0(t,b)D^{\alpha}\left[\frac{pW(x,y)(t)}{e_0(t,b)}\right] = e_0(t,b)D^{\alpha}\left[\frac{ce^2_0(t,b)}{e_0(t,b)}\right]\equiv 0; \]
since $Lx=0$, this yields $x(Ly)(t) \equiv 0$ for $t\in\mathcal{I}$. As $x\neq 0$ on $\mathcal{I}$, $(Ly)(t)\equiv 0$ on $\mathcal{I}$. Thus $y$ is also a solution of $Lx=0$.
\end{proof}

Before we end this section, we include a result on the Wronskian for solutions of the homogeneous version of the general second-order conformable proportional differential equation \eqref{lineq}.

% Liouville's Theorem for general 2nd order %

\begin{theorem}[Liouville's Theorem]\label{liouville}
Assume $\kappa_0,\kappa_1$ satisfy \eqref{kappaconditions}. Let $\alpha\in(0,1]$, and let $D^{\alpha}$ be as in \eqref{derivdef}.
In \eqref{lineq} assume $a,b,c$ are continuous on $\mathcal{I}$ with $a(t)\ne 0$, and $g\equiv 0$. If $x,y$ are solutions of the resulting general homogeneous equation
\begin{equation}\label{gen2ndo}
 a(t)D^{\alpha}D^{\alpha}x(t)+b(t)D^{\alpha}x(t)+c(t)x(t)=0, \quad t\in\mathcal{I},
\end{equation}
then their Wronskian \eqref{wdef} satisfies Liouville's formula
\[ W(x,y)(t) = W(x,y)(t_0)e_{(-b/a-\kappa_1)}(t,t_0) \]
for all $t,t_0\in\mathcal{I}$. 
\end{theorem}

\begin{proof}
If $x,y$ are solutions of \eqref{gen2ndo}, then the Wronskian of $x$ and $y$ satisfies
\begin{eqnarray*}
 D^{\alpha}W &=& D^{\alpha}[xD^{\alpha}y-yD^{\alpha}x] \\
 &=& xD^{\alpha}D^{\alpha}y + (D^{\alpha}x)(D^{\alpha}y) -\kappa_1xD^{\alpha}y - yD^{\alpha}D^{\alpha}x - (D^{\alpha}x)(D^{\alpha}y) +\kappa_1yD^{\alpha}x \\
 &=& x\left(-\frac{b}{a}D^{\alpha}y-\frac{c}{a}y\right)-y\left(-\frac{b}{a}D^{\alpha}x-\frac{c}{a}x\right)-\kappa_1W \\
 &=& -\frac{b}{a}W-\kappa_1W
\end{eqnarray*}
on $\mathcal{I}$. By \eqref{expderiv}, we see that $W(x,y)(t) = W(x,y)(t_0)e_{(-b/a-\kappa_1)}(t,t_0)$, and the result holds.
\end{proof}

%%%%%%%%%%%%%%%%%%%%%%%%%%%%%%%%
% SECTION 4 Reduction of Order %
%%%%%%%%%%%%%%%%%%%%%%%%%%%%%%%%

\section{Reduction of Order Theorems}\label{secreduct}

In this section we establish two related reduction of order theorems for the conformable self-adjoint equation $Lx=0$ for $L$ given in \eqref{saeq}, and one for the general second-order equation.

%%%%%%%%%%%%%%%%%%%%%
% Reduction Order I %
%%%%%%%%%%%%%%%%%%%%%

\begin{theorem}[Reduction of Order I]\label{order}
Let $t_0\in\mathcal{I}$, and assume $x$ is a solution of $Lx=0$ for $L$ given in \eqref{saeq} with $x\neq 0$ on $\mathcal{I}$. Then 
\[ y(t):= x(t)\int_{t_0}^t \frac{e^2_{0}(s,t_0)}{p(s)x^2(s)} d_{\alpha}s, \quad t\in\mathcal{I} \]
defines a second linearly independent solution $y$ of $Lx=0$ on $\mathcal{I}$.
\end{theorem}

\begin{proof}
For $y$ defined above and $x\neq 0$ on $\mathcal{I}$, by the product rule in Lemma \ref{basicderiv} and \cite[Lemma 1.9 (iv)]{au} we have
\begin{equation}\label{dalphay}
 D^{\alpha}y = x\left[\frac{e^2_0(t,t_0)}{px^2} + \kappa_1\frac{y}{x}\right] + (D^{\alpha}x)\frac{y}{x} - \kappa_1y = \frac{e^2_0(t,t_0)}{px} + \frac{y}{x}(D^{\alpha}x), 
\end{equation}
which is continuous on $\mathcal{I}$ since $x\in\D$, $x\neq 0$, and $p>0$ is continuous. For the conformable Wronskian $W$ given in Definition \ref{wdef}, using \eqref{dalphay} we have
\begin{equation}\label{rooW}
 W(x,y) =  xD^{\alpha}y - yD^{\alpha}x =  e^2_0(t,t_0)/p + yD^{\alpha}x - yD^{\alpha}x = e^2_0(t,t_0)/p. 
\end{equation}
Now, continuing from \eqref{dalphay} we have
\[ pD^{\alpha}y = \frac{e^2_0(\cdot,t_0)}{x} + \frac{y}{x}p(D^{\alpha}x), \]
from which we see that
\begin{eqnarray*}
 D^{\alpha}[pD^{\alpha}y] 
    &=& D^{\alpha}[e^2_0/x + p(D^{\alpha}x)y/x] \\
    &=& D^{\alpha}[e^2_0/x] + p(D^{\alpha}x)D^{\alpha}[y/x] + D^{\alpha}[pD^{\alpha}x]y/x - \kappa_1 p(D^{\alpha}x)y/x \\
    &=& D^{\alpha}[e^2_0/x] + p(D^{\alpha}x)\left[\frac{e^2_{0}}{px^2} + \kappa_1y/x\right] + D^{\alpha}[pD^{\alpha}x]y/x - \kappa_1 p(D^{\alpha}x)y/x \\
    &=& D^{\alpha}[e^2_0/x] + \frac{e^2_{0}}{x^2}(D^{\alpha}x) + D^{\alpha}[pD^{\alpha}x]y/x \\
    &=& - \frac{e^2_0}{x^2}(D^{\alpha}x) + \frac{e^2_{0}}{x^2}(D^{\alpha}x) + D^{\alpha}[pD^{\alpha}x]y/x \\
    &=& -qy.
\end{eqnarray*}
Therefore $y\in\D$, and by Theorem \ref{abelcon}, $y$ is also a linearly independent solution of $Lx=0$.
\end{proof}

%%%%%%%%%%%%%%%%%%%%%%
% Reduction Order II %
%%%%%%%%%%%%%%%%%%%%%%

\begin{theorem}[Reduction of Order II]\label{order2}
Let $L$ be as in \eqref{saeq}, let $t_0\in\mathcal{I}$, and assume $x$ is a solution of $Lx=0$ with $x\neq 0$ on $\mathcal{I}$. Then $y$ is a second linearly independent solution of $Lx=0$ iff $y$ satisfies the first-order equation
\begin{equation}\label{401}
 D^{\alpha}[y/x](t) - \kappa_1(t) [y/x](t) = \frac{ce^2_0(t,t_0)}{p(t)x^2(t)}, \quad t\in\mathcal{I}, \quad t\ge t_0, 
\end{equation}
for some constant $c\in\R$ iff $y$ is of the form
\begin{equation}\label{402}
 y(t) = c_1x(t) + c_2x(t)\int_{t_0}^t \frac{e^2_{0}(s,t_0)}{p(s)x^2(s)} d_{\alpha}s
\end{equation}
for $t\in\mathcal{I}$ with $t\ge t_0$, where $c_1,c_2\in\R$ are constants. In the latter case, 
\begin{equation}\label{403}
 c_1=y(t_0)/x(t_0), \qquad c_2=p(t_0)W(x,y)(t_0).
\end{equation}
\end{theorem}

\begin{proof}
Assume $x$ is a solution of $Lx=0$ with $x\neq 0$ on $\mathcal{I}$. Let $y$ be any solution of $Lx=0$; we must show $y$ is of the form \eqref{402}. Using the Wronskian from Definition \ref{wdef}, set
\[ c_2:=p(t_0)W(x,y)(t_0) = p(t_0)(xD^{\alpha}y - yD^{\alpha}x)(t_0). \]
By Abel's formula, Corollary \ref{wronskid}, we must have 
\[ W(x,y)(t) = \frac{c_2e^2_0(t,t_0)}{p(t)}, \quad t\in\mathcal{I}, \]
so that
\[ \frac{xD^{\alpha}y - y(D^{\alpha}x)}{x^2} +\kappa_1\frac{y}{x} = \frac{c_2e^2_0(t,t_0)}{px^2}+\kappa_1\frac{y}{x}. \]
Thus, we see that
\begin{equation}\label{dxinvu} 
  D^{\alpha}[y/x] = \frac{c_2e^2_0(t,t_0)}{px^2}+\kappa_1\frac{y}{x} 
\end{equation}
and $y$ satisfies \eqref{401}. Integrating both sides of \eqref{401} from $t_0$ to $t$ we get
\[ y(t)/x(t)- y(t_0)/x(t_0) = c_2\int_{t_0}^t \frac{e^2_{0}(s,t_0)}{p(s)x^2(s)} d_{\alpha}s; \]
recovering $y$ yields
\[ y(t) = c_1x(t) + c_2x(t)\int_{t_0}^t \frac{e^2_{0}(s,t_0)}{p(s)x^2(s)} d_{\alpha}s \]
provided $c_1=y(t_0)/x(t_0)$. 

Conversely, assume $y$ is given by \eqref{402}. By Theorem \ref{order} and linearity $y$ is a solution of $Lx=0$ on $\mathcal{I}$ for $t\ge t_0$.  Setting $t=t_0$ in \eqref{402} leads to $c_1$ in \eqref{403}. By Abel's formula, Corollary \ref{wronskid}, the Wronskian satisfies 
\[ W(x,y)(t) = \frac{e^2_0(t,t_0)p(t_0)W(x,y)(t_0)}{p(t)}; \] 
to calculate $W(x,y)(t_0)$, we use \eqref{402} to obtain
\[ W(x,y)(t_0) = (xD^{\alpha}y - yD^{\alpha}x)(t_0) = \frac{c_2}{p(t_0)}. \]
This ends the proof.
\end{proof}

In a similar way we can prove a reduction of order theorem for the general second-order homogeneous equation.

%%%%%%%%%%%%%%%%%%%%%%%
% Reduction Order iii %
%%%%%%%%%%%%%%%%%%%%%%%

\begin{theorem}[Reduction of Order]\label{rooiii}
Let $t_0\in\mathcal{I}$, and let the coefficient functions $a,b,c$ be continuous on $\mathcal{I}$ with $a(t)\ne 0$ for all $t\in\mathcal{I}$. Further, assume $x$ is a solution of the general equation \eqref{gen2ndo} with $x\ne 0$ on $\mathcal{I}$. Then 
\[ y(t):= x(t)\int_{t_0}^t \frac{e_{(-b/a-\kappa_1)}(s,t_0)}{x^2(s)} d_{\alpha}s, \quad t\in\mathcal{I} \]
defines a second linearly independent solution $y$ of \eqref{gen2ndo} on $\mathcal{I}$.
\end{theorem}

%%%%%%%%%%%%%%%%%%%%%%
% SECTION 5 Cauchy F %
%%%%%%%%%%%%%%%%%%%%%%

\section{Cauchy Function and Variation of Constants Formula}

In this section discuss the Cauchy function and derive a variation of constants formula for the conformable nonhomogeneous self-adjoint differential equation $Lx=h$ for $L$ given in \eqref{saeq}, where we assume $h$ is a continuous function on some interval $\mathcal{I}\subseteq\R$. The following theorem follows from a standard argument, using the linearity of $D^{\alpha}$ in \eqref{derivdef} and Theorem \ref{unique}.

% Theorem: General Solution %

\begin{theorem}
If $x_1$ and $x_2$ are linearly independent solutions of the conformable homogeneous equation $Lx=0$ on $\mathcal{I}$, and $y$ is a particular solution of the conformable nonhomogeneous equation $Lx=h$ on $\mathcal{I}$, then 
\[ x = c_1x_1 + c_2x_2 + y \]
is a general solution of $Lx=h$ for constants $c_1,c_2\in\R$.
\end{theorem}

The following definition was given in \cite[Definition 4.11]{au}, and is included here for completeness.

%%%%%%%%%%%%%%%%
% DEFINITION   %
%%%%%%%%%%%%%%%%

\begin{definition}[Cauchy Function]
Let $L$ be as in \eqref{saeq}. A function $x:\mathcal{I}\times\mathcal{I}\to\R$ is the Cauchy function for $Lx=0$ provided for each fixed $s\in\mathcal{I}$, $x(\cdot,s)$ is the solution of the initial value problem
\[ Lx(\cdot,s) = 0,\quad x(s,s) = 0, \quad D^{\alpha}x(s,s) = \frac{1}{p(s)}. \]
\end{definition}

It is easy to verify the following example.

%%%%%%%%%%%%%
% EXAMPLE   %
%%%%%%%%%%%%%

\begin{example}\label{oldex1}
If $q,h\equiv 0$ in \eqref{saeq}, then the Cauchy function for 
\[ D^{\alpha}\left[pD^{\alpha}x\right](t) = 0 \] 
is given by 
\[ x(t,s) = e_0(t,s)\int_{s}^t \frac{1}{p(\tau)}d_{\alpha}\tau \]
for all $t,s\in\mathcal{I}$. $\hfill\triangle$
\end{example}

For $L$ in \eqref{saeq}, a formula for the Cauchy function for $Lx=0$ is given in the next theorem, whose proof is very similar to that given for a different equation in \cite[Theorem 4.13]{au}, and thus the proof is omitted.

%%%%%%%%%%%%%
% THEOREM   %
%%%%%%%%%%%%%

\begin{theorem}\label{t33} \index{Cauchy function!second order equation}
If $u$ and $v$ are linearly independent solutions of $Lx=0$ for $L$ in \eqref{saeq}, then the Cauchy function $x(t,s)$ for $Lx=0$ is given by
\begin{equation}\label{eq31}
  x(t,s) = \frac{u(s)v(t)-v(s)u(t)}{p(s)[u(s)D^{\alpha}v(s)-v(s)D^{\alpha}u(s)]} \quad\text{for}\quad t,s\in\mathcal{I}.
\end{equation}
\end{theorem}

%%%%%%%%%%%%%
% THEOREM   %
%%%%%%%%%%%%%

\begin{theorem}[Variation of Constants Formula]\label{t31}
Assume $h$ is continuous on $\mathcal{I}$ and $a\in\mathcal{I}$. Let $x(t,s)$ be the Cauchy function for $Lx=0$ for $L$ in \eqref{saeq}. Then 
\begin{equation}\label{vofcform}
 x(t) = \int_a^t x(t,s)h(s) d_{\alpha}s, \quad t\in\mathcal{I} 
\end{equation}
is the solution of the initial value problem
\[ Lx = h(t), \quad x(a) = 0, \quad D^{\alpha}x(a) = 0. \]
\end{theorem}

\begin{proof} 
Let $x(t,s)$ be the Cauchy function for $Lx=0$, and set 
\[ x(t) = \int_a^t x(t,s)h(s) d_{\alpha}s. \]
Note that $x(a)=0$. Taking the conformable derivative $D^{\alpha}$ of $x$ and using \cite[Lemma 1.9 (iv)]{au}, we get that
\begin{equation}\label{vofcderiv}
 D^{\alpha}x(t) = x(t,t)h(t) + \int_a^t D^{\alpha}[x(t,s)]h(s)d_{\alpha}s = \int_a^t D^{\alpha}[x(t,s)]h(s)d_{\alpha}s,
\end{equation}
since the Cauchy function satisfies $x(t,t)=0$. Note that in the integral, $D^{\alpha}$ denotes the derivative with respect to the first variable $t$;
thus $D^{\alpha}x(a)=0$. From \eqref{vofcform} and \eqref{vofcderiv},
\begin{equation}\label{pswitch}
 pD^{\alpha}x(t) = \int_a^t p(t)D^{\alpha}[x(t,s)]h(s)d_{\alpha}s, 
\end{equation}
and we conclude, using \cite[Lemma 1.9 (iv)]{au} again, that
\begin{eqnarray*}
 D^{\alpha}\left[pD^{\alpha}x\right](t)
 &=& p(t)D^{\alpha}[x(t,t)]h(t) + \int_a^t D^{\alpha}\left[p(t)D^{\alpha}x(t,s)\right]h(s)d_{\alpha}s \\
 &=& h(t) - \int_a^t q(t)x(t,s)h(s)d_{\alpha}s \\
 &=& h(t) - q(t)x(t),
\end{eqnarray*}
by all of the properties of the Cauchy function. Consequently, $Lx(t) = h(t)$.
\end{proof}

%%%%%%%%%%%%%%%%%%%%%
% SECTION 6 RICCATI %
%%%%%%%%%%%%%%%%%%%%%

\section{Riccati equation}\label{riccatieq}

Now we will consider the associated conformable Riccati differential equation
\begin{equation}\label{ricceq}
	Rz=0, \quad\text{where}\quad Rz:=D^{\alpha}z + q + \frac{z^2}{p} - \kappa_1z,
\end{equation}
where $p$ and $q$ are continuous functions on $\mathcal{I}$ such that $p(t)>0$ for all $t\in\mathcal{I}$.

% Riccati Definition %

\begin{definition}
Denote by $\D_R$ the set of all differentiable functions $z$ such that $D^{\alpha}z(t)$ is continuous for all $t\in\mathcal{I}$. A function $z\in\D_R$ is a solution of \eqref{ricceq} on $\mathcal{I}$ if and only if $Rz(t)=0$ for all $t\in\mathcal{I}$.
\end{definition}

% Riccati Example %

\begin{example}
Let $\alpha\in(0,1]$, and let $\kappa_0,\kappa_1$ satisfy \eqref{kappaconditions}. Furthermore, let $\theta_-,\theta_+\in\R$ with $\theta_-<0<\theta_+$ such that 
\[ h_1(\theta_-,0):= \int_{0}^{\theta_-}1d_{\alpha}s = -\frac{\pi}{2}, \quad h_1(\theta_+,0):= \int_{0}^{\theta_+}1d_{\alpha}s = \frac{\pi}{2}. \]
Then the solution of the initial value problem
\[ Rz=0, \quad Rz:=D^{\alpha}z + 1 + z^2 - \kappa_1z, \quad z(0)=0 \]
is given by
\[ z(t)=-\tan\left(h_1(t,0)\right), \quad t\in\mathcal{I}:=\left(\theta_-,\theta_+\right). \]
Here we have solved \eqref{ricceq} with $p\equiv q\equiv 1$. $\hfill\triangle$
\end{example}

%%%%%%%%%%%%%%%%%%%%%%%%%%
% Self-Adjoint / Riccati %
%%%%%%%%%%%%%%%%%%%%%%%%%%

\begin{theorem}[Factorization Theorem]\label{factorricc}
Let $L$ be as in \eqref{saeq}, $R$ as in \eqref{ricceq}, and let $x\in\D$ with $x\neq 0$ on $\mathcal{I}$. If $z$ is defined by
\begin{equation}\label{zform}
 z = \frac{pD^{\alpha}x}{x}
\end{equation}
on $\mathcal{I}$, then $Lx=xRz$ on $\mathcal{I}$.
\end{theorem}

\begin{proof} 
Assume $x\in\D$ with $x\neq 0$ on $\mathcal{I}$, and let $z$ have the form \eqref{zform}. Then
\begin{eqnarray*} 
 Lx &=& D^\alpha\left[pD^\alpha x\right] + qx \stackrel{\eqref{zform}}{=} D^\alpha[zx] + qx = zD^{\alpha}x + xD^{\alpha}z -\kappa_1zx + qx \\
 &\stackrel{\eqref{zform}}{=}& z\left(\frac{zx}{p}\right) + xD^{\alpha}z - \kappa_1zx + qx \stackrel{\eqref{ricceq}}{=} xRz
\end{eqnarray*}
on $\mathcal{I}$.
\end{proof}

%%%%%%%%%%%%%%
% Definition %
%%%%%%%%%%%%%%

\begin{definition}\label{disconj}
A solution $x$ of 
\begin{equation}\label{saveq}
 Lx(t) = D^\alpha[pD^\alpha x](t) + q(t)x(t) = 0, \quad t\in\mathcal{I},
\end{equation}
has a zero at $t_0\in\mathcal{I}$ iff $x(t_0)=0$. Equation \eqref{saveq} is disconjugate on $\mathcal{I}$ iff no nontrivial  solution has two (or more) zeros on $\mathcal{I}$.
\end{definition}

% Theorem %

\begin{theorem}\label{saricc}
Let $L$ be as in \eqref{saeq}, $R$ as in \eqref{ricceq}. Then the self-adjoint equation $Lx=0$ has a solution without zeros on $\mathcal{I}$ if and only if the Riccati equation $Rz=0$ has a solution $z$ on $\mathcal{I}$. Moreover, these two solutions satisfy \eqref{zform} on $\mathcal{I}$.
\end{theorem}

\begin{proof} 
First assume $Lx=0$ has a solution $x\in\D$ with $x\neq 0$ on $\mathcal{I}$, and let $z$ have the form \eqref{zform}. By Theorem \ref{factorricc} we have
$xRz=Lx=0$ on $\mathcal{I}$, making $z$ a solution of \eqref{ricceq} on $\mathcal{I}$.

Conversely, let $z$ be a solution of \eqref{ricceq} on $\mathcal{I}$. Then $(z/p)$ is continuous, and we let 
\[ x = e_{(z/p)}(t,t_{0}), \quad t_0\in\mathcal{I}. \]  
Then $x$ is well defined and never zero.  Note also that $D^{\alpha}x = zx/p$ is continuous, and if we solve this equation for $z$ we get \eqref{zform}. Furthermore,
\[ D^{\alpha}[pD^{\alpha}x] \stackrel{\eqref{zform}}{=} D^{\alpha}[zx] = zD^{\alpha}x + xD^{\alpha}z - \kappa_1zx \] 
is continuous on $\mathcal{I}$, putting $x\in\D$. Moreover,
\[ D^{\alpha}[pD^{\alpha}x] = zD^{\alpha}x + x(D^{\alpha}z - \kappa_1z) = z^2x/p + x(-q-z^2/p) = -qx. \]
This shows that $x$ is a solution of $Lx=0$ on $\mathcal{I}$, completing the proof.
\end{proof}

%%%%%%%%%%%%%%%%%%%%%%%%
% Section 7 functional %
%%%%%%%%%%%%%%%%%%%%%%%%

\section{Quadratic Functional and Picone Identities}\label{secquad}

%%%%%%%%%%%%%%
% admissible %
%%%%%%%%%%%%%%

We are now interested in exploring the connection between the disconjugacy (Definition \ref{disconj}) of the equation $Lx=0$ on $\mathcal{I}$ for $L$ in \eqref{saveq}, and the positive definiteness of a certain related quadratic functional. To this end, for $a,b\in\mathcal{I}$, we define the set of admissible functions $\A$ to be
\[ \A:=\left\{\eta\in \text{C}^1([a,b],\R)\backslash\{0\}:\eta(a)=\eta(b)=0\right\}. \]
Here C$^1([a,b],\R)$ denotes the set of all continuous functions whose regular derivatives are piece-wise continuous. Then we define the quadratic functional $\cF$ on $\A$ via
\begin{eqnarray}\label{quadF}  
 \cF(\eta) &=& \int_a^b \Big[p(D^{\alpha}\eta)^2 - q\eta^2 \Big](t) e^2_0(b,t)d_{\alpha}t.
\end{eqnarray}

%%%%%%%%%%%%%%%%%%%%%%%%%%%%%%%%%
% Definition: Positive Definite %
%%%%%%%%%%%%%%%%%%%%%%%%%%%%%%%%%

\begin{definition}
The functional $\cF$ is positive definite on $\A$ (write $\cF>0$) provided $\cF(\eta)\ge 0$ for all $\eta\in\A$, and $\cF(\eta)=0$ iff $\eta=0$.
\end{definition}

%%%%%%%%%%%%%
% Lemma 511 %
%%%%%%%%%%%%%

\begin{lemma}\label{lemma511}
Let $u$ be a solution of \eqref{saveq} on $[a,b]$, and let $c,d\in[a,b]$ with $a \le c < d \le b$. Then for
\[ \eta(t):=\begin{cases} u(t) & : t\in(c,d) \\ 0 & : t\not\in(c,d) \end{cases} \]
we have $\eta\in\A$ and
\[ \cF(\eta) = \frac{[puD^{\alpha}u](d)}{e^2_0(d,b)} - \frac{[puD^{\alpha}u](c)}{e^2_0(c,b)}. \]
\end{lemma}

\begin{proof}
Let $u,\eta$ be as described in the statement of the lemma. It is apparent that $\eta\in \text{C}^1([a,b],\R)$ with $\eta(a)=0$ and $\eta(b)=0$, putting $\eta\in\A$. Now consider $\cF(\eta)$ for $\cF$ in \eqref{quadF}. We have 
\[ \eta = 0 = D^{\alpha}\eta \quad\text{on}\quad [a,c)\cup(d,b], \]
so
\[ \cF(\eta) = \int_c^d \Big[p(D^{\alpha}u)^2 - qu^2 \Big](t) e^2_0(b,t)d_{\alpha}t, \]
giving us
\[ \cF(\eta) = \int_c^d (e_0(b,d)e_0(b,t)pD^{\alpha}u)(D^{\alpha}u)(t) e_0(d,t)d_{\alpha}t - \int_c^d (qu^2)(t) e^2_0(b,t)d_{\alpha}t. \]
Using the integration by parts formula \cite[Lemma 1.9 (iii)]{au}
\[ \int_c^d D^{\alpha}[f(t)]g(t) e_0(d,t)d_{\alpha}t = f(t)g(t)e_0(d,t)\big|_{t=c}^{d}  
   - \int_c^d f(t)\left(D^{\alpha}[g]-\kappa_1g\right)(t) e_0(d,t)d_{\alpha}t \]
with $u=f$ and $(e_0(b,d)e_0(b,t)pD^{\alpha}u)=g$ yields
\begin{eqnarray*}
 \cF(\eta) &=& \frac{puD^{\alpha}u}{e^2_0(t,b)}\Big|_c^d - \int_c^d u\left(D^{\alpha}\left[\frac{pD^{\alpha}u}{e_0(\cdot,b)}\right] 
               -\kappa_1\frac{pD^{\alpha}u}{e_0(\cdot,b)}\right)(t)e_0(b,t)d_{\alpha}t - \int_c^d (qu^2)(t) e^2_0(b,t)d_{\alpha}t \\
           &=& \frac{puD^{\alpha}u}{e^2_0(t,b)}\Big|_c^d - \int_c^d \frac{u(t)}{e_0(t,b)}D^{\alpha}\left[e_0(\cdot,b)\frac{pD^{\alpha}u}{e_0(\cdot,b)}\right]e_0(b,t)d_{\alpha}t
               - \int_c^d (qu^2)(t) e^2_0(b,t)d_{\alpha}t,
\end{eqnarray*}
where we have used the equality
\begin{equation}\label{dfminuskf} 
  D^{\alpha} f(t) - \kappa_1(t) f(t) = \frac{1}{e_0(t,b)}D^{\alpha}\left[e_0(\cdot,b)f\right](t).
\end{equation}
Therefore,
\begin{eqnarray*}
 \cF(\eta) &=& \frac{puD^{\alpha}u}{e^2_0(t,b)}\Big|_c^d - \int_c^d u(t)\left(D^{\alpha}[pD^{\alpha}u]+qu\right)(t)e^2_0(b,t)d_{\alpha}t;
\end{eqnarray*}
as $u$ is a solution of \eqref{saveq}, we have
\[ Lu = D^{\alpha}[pD^{\alpha}u] + qu = 0. \]
Consequently,
\[ \cF(\eta) = \frac{puD^{\alpha}u}{e^2_0(t,b)}\Big|_c^d - \int_c^d u(t)\left(Lu\right)(t)e^2_0(b,t)d_{\alpha}t  = \frac{puD^{\alpha}u}{e^2_0(t,b)}\Big|_c^d, \]
and the proof is complete.
\end{proof}

%%%%%%%%%%%%%%%%%%%%%%%%%%%%%
% Theorem Picone Identity 1 %
%%%%%%%%%%%%%%%%%%%%%%%%%%%%%

\begin{theorem}[Picone Identity I]\label{picone}
Assume the Riccati equation \eqref{ricceq} has a solution $z$ on $[a,b]$. Then the conformable self-adjoint equation $Lx=0$ for $L$ in \eqref{saeq} has a solution $x$ that is nonzero on $[a,b]$, and for any admissible function $\eta\in\A$, on $[a,b]$ we have the equality
\[ D^{\alpha}[z\eta^2] + \kappa_1z\eta^2 = p(D^{\alpha}\eta)^2 - q\eta^2 - p(D^{\alpha}\eta - z\eta/p)^2. \]
\end{theorem}

\begin{proof}
The first claim follows from the assumptions and Theorem \ref{saricc}. For any $\eta\in\A$, we have on $[a,b]$ that
\begin{eqnarray*}
 D^{\alpha}[z\eta^2] + \kappa_1z\eta^2 
 &=& \eta D^{\alpha}[z\eta] + z\eta  D^{\alpha}[\eta] \\
 &=& \eta[zD^{\alpha}\eta + \eta  D^{\alpha}z - \kappa_1 z\eta] + z\eta  D^{\alpha}\eta \\
 &\stackrel{\eqref{ricceq}}{=}& 2z\eta  D^{\alpha}\eta + \eta^2(-q-z^2/p) \\
 &=& p(D^{\alpha}\eta)^2 - p(D^{\alpha}\eta)^2 + 2z\eta  D^{\alpha}\eta - q\eta^2 - pz^2\eta^2/p^2 \\
 &=& p(D^{\alpha}\eta)^2 - q\eta^2 - p\left(D^{\alpha}\eta-\frac{z\eta}{p}\right)^2, 
\end{eqnarray*}
and the result follows.
\end{proof}

% Remark on 2nd Picone Identity %

\begin{remark}
Another Picone identity is possible, this one for the system of self-adjoint conformable equations given by
\begin{eqnarray}
\label{piconeq1} D^\alpha[p_1D^\alpha x](t) + q_1(t)x(t) &=& 0, \\
\label{piconeq2} D^\alpha[p_2D^\alpha x](t) + q_2(t)x(t) &=& 0,
\end{eqnarray}
where $p_j$ and $q_j$ are continuous functions on the interval $\mathcal{I}\subseteq\R$ with $p_j(t)>0$ for all $t\in\mathcal{I}$, for $j=1,2$. 
\end{remark}

%%%%%%%%%%%%%%%%%%%%%%%%%%%%%
% Theorem Picone Identity 2 %
%%%%%%%%%%%%%%%%%%%%%%%%%%%%%

\begin{theorem}[Picone Identity II]\label{picone2}
Let $p_j$ and $q_j$ be continuous functions on the interval $\mathcal{I}\subseteq\R$ with $p_j(t)>0$ for all $t\in\mathcal{I}$, for $j=1,2$. Assume $u$ and $v$ are solutions of \eqref{piconeq1} and \eqref{piconeq2}, respectively, with $v(t)\ne 0$ on $[a,b]\subset\mathcal{I}$. Then we have the equality
\begin{equation}\label{picone2eq}
 D^{\alpha}\left[\frac{u}{v}\left(p_1vD^{\alpha}u-p_2uD^{\alpha}v\right)\right] 
 = (q_2 - q_1)u^2 + (p_1 - p_2)(D^{\alpha}u)^2 + p_2\left(D^{\alpha}u - \frac{uD^{\alpha}v}{v}\right)^2.
\end{equation}
\end{theorem}

\begin{proof}
Assume $u$ and $v$ are solutions of \eqref{piconeq1} and \eqref{piconeq2}, respectively, with $v(t)\ne 0$ on $[a,b]\subset\mathcal{I}$.
For $t\in[a,b]\subset\mathcal{I}$, we use the alpha product rule on $\frac{u}{v}\cdot\left(vp_1D^{\alpha}u-up_2D^{\alpha}v\right)$ to see that 
\begin{eqnarray*}
 && D^{\alpha}\left[\frac{u}{v}\left(vp_1D^{\alpha}u-up_2D^{\alpha}v\right)\right]\!+\!\kappa_1\frac{u}{v}\left(vp_1D^{\alpha}u-up_2D^{\alpha}v\right) \\
 && =\frac{u}{v}\left\{vD^{\alpha}[p_1D^{\alpha}u]+p_1D^{\alpha}uD^{\alpha}v-uD^{\alpha}[p_2D^{\alpha}v]-p_2D^{\alpha}uD^{\alpha}v\right\} 
    + D^{\alpha}\left[\frac{u}{v}\right]\left(vp_1D^{\alpha}u-up_2D^{\alpha}v\right) \\
 && =uD^{\alpha}[p_1D^{\alpha}u]-\frac{u^2}{v}D^{\alpha}[p_2D^{\alpha}v]+\frac{u}{v}(p_1-p_2)D^{\alpha}uD^{\alpha}v \\
 &&  \hspace{1in} + \left(\frac{D^{\alpha}u}{v}-\frac{uD^{\alpha}v}{v^2} + \frac{\kappa_1u}{v}\right)(vp_1D^{\alpha}u - up_2D^{\alpha}v).
\end{eqnarray*} 
Since $u$ and $v$ are solutions of \eqref{piconeq1} and \eqref{piconeq2}, respectively, we can simplify both sides of the expression above to get
\begin{eqnarray*}
 && D^{\alpha}\left[\frac{u}{v}\left(vp_1D^{\alpha}u-up_2D^{\alpha}v\right)\right] \\
 && = (q_2-q_1)u^2 + \frac{u}{v}(p_1-p_2)D^{\alpha}uD^{\alpha}v + \left(D^{\alpha}u-\frac{uD^{\alpha}v}{v}\right)\left(p_1D^{\alpha}u 
       - \frac{u}{v}p_2D^{\alpha}v\right) \\
 && = (q_2-q_1)u^2 + (p_1-p_2)(D^{\alpha}u)^2 + p_2\left((D^{\alpha}u)^2-\frac{2u}{v}D^{\alpha}uD^{\alpha}v+\frac{u^2(D^{\alpha}v)^2}{v^2}\right).
\end{eqnarray*} 
Therefore, equality \eqref{picone2eq} holds and the proof is complete.
\end{proof}

Using the Picone Identity II from Theorem \ref{picone2} above, we prove the following conformable Sturm comparison result.

% Sturm Comparison Theorem %

\begin{theorem}[Sturm Comparison Theorem]\label{sct}
Let $p_j$ and $q_j$ be continuous functions on the interval $\mathcal{I}\subseteq\R$ with $p_j(t)>0$ for all $t\in\mathcal{I}$, for $j=1,2$. Assume $u$ is a solution of \eqref{piconeq1} with consecutive zeros at $a<b$ in $\mathcal{I}$, and assume that
\begin{equation}\label{kp522}
 q_2(t) \ge q_1(t), \quad p_1(t) \ge p_2(t) > 0, \quad t\in[a,b].
\end{equation}
Let $v$ be a nontrivial solution of \eqref{piconeq2}. If for some $t\in[a,b]$ one of the inequalities in \eqref{kp522} is strict, or if $u$ and $v$ are linearly independent on $[a,b]$, then $v$ has a zero in $(a,b)$. 
\end{theorem}

\begin{proof}
Let $u$ be a solution of \eqref{piconeq1} with consecutive zeros at $a<b$ in $\mathcal{I}$, but assume the conclusion of the theorem does not hold. To wit, assume $v$ is a solution of \eqref{piconeq2} with 
\[ v(t) \ne 0, \quad t\in(a,b). \]
Assume $a<c<d<b$, and multiply Picone Identity II in \eqref{picone2eq} by $e_0(d,t)$ to get
\begin{eqnarray*}
 && D^{\alpha}\left[\frac{u}{v}\left(p_1vD^{\alpha}u-p_2uD^{\alpha}v\right)\right](t)e_0(d,t) \\ 
 && = \left\{(q_2 - q_1)u^2 + (p_1 - p_2)(D^{\alpha}u)^2 + p_2\left(D^{\alpha}u - \frac{uD^{\alpha}v}{v}\right)^2\right\}(t)e_0(d,t). 
\end{eqnarray*}
Integrating this from $c$ to $d$ and employing Theorem \ref{ftc} we obtain
\begin{eqnarray}
\nonumber & & \left[\frac{u}{v}\left(p_1vD^{\alpha}u-p_2uD^{\alpha}v\right)\right](t)e_0(d,t)\Big|_{t=c}^d \\
\label{kp522int} & &  = \int_c^d \left\{(q_2 - q_1)u^2 + (p_1 - p_2)(D^{\alpha}u)^2 + p_2\left(D^{\alpha}u - \frac{uD^{\alpha}v}{v}\right)^2\right\}(t)e_0(d,t)d_{\alpha}t.
\end{eqnarray}
From the inequalities in \eqref{kp522} and the assumption that for some $t\in[a,b]$ one of the inequalities in \eqref{kp522} is strict, or that $u$ and $v$ are linearly independent on $[a,b]$, we take limits in \eqref{kp522int} as $c\rightarrow a^+$ and $d\rightarrow b^-$ to see that
\[ \lim_{c\rightarrow a^+, \;d\rightarrow b^-} \left[\frac{u}{v}\left(p_1vD^{\alpha}u-p_2uD^{\alpha}v\right)\right](t)e_0(d,t)\Big|_{t=c}^d >0. \]
We will thus arrive at a contradiction if these two limits are zero. We will prove only the case where $c\rightarrow a^+$, as the other case is similar and thus omitted. Now if $v(a)\ne 0$, then 
\[ \lim_{c\rightarrow a^+} \left[\frac{u}{v}\left(p_1vD^{\alpha}u-p_2uD^{\alpha}v\right)\right](c)e_0(d,c) = 0 \]
as $u(a)=0$. If $v(a)=0$, however, then $D^{\alpha}v(a)\ne 0$, since by Theorem \ref{unique}, only the trivial solution satisfies \eqref{piconeq2} with initial conditions $x(a)=0$, $D^{\alpha}x(a)=0$, and $v$ is nontrivial. Consequently,
\begin{eqnarray*} 
 & & \lim_{c\rightarrow a^+} \left[\frac{u}{v}\left(p_1vD^{\alpha}u-p_2uD^{\alpha}v\right)\right](c)e_0(d,c)
     = -e_0(d,a)\lim_{c\rightarrow a^+} \left[\frac{u^2p_2D^{\alpha}v}{v}\right](c) \\
 & & \stackrel{\text{L'H}}= -e_0(d,a)\lim_{c\rightarrow a^+} 
     \left[\frac{u^2D^{\alpha}[p_2D^{\alpha}v]+(p_2D^{\alpha}v)D^{\alpha}[u^2]-\kappa_1u^2p_2D^{\alpha}v}{D^{\alpha}v}\right](c) \\
 & & = -e_0(d,a)\lim_{c\rightarrow a^+}
     \left[\frac{-q_2uv + 2p_2D^{\alpha}uD^{\alpha}v - 2\kappa_1up_2D^{\alpha}v}{D^{\alpha}v}\right](c)u(c) \\
 & & = 0,
\end{eqnarray*}
using L'H\^{o}pital's Rule. This rule is valid in this case; suppose $f$ and $g$ are continuously differentiable at $a$ such that $f(a)=0=g(a)$ with $D^{\alpha}g(a)\ne 0$. Then
\[ \lim_{c\rightarrow a} \frac{f(c)}{g(c)} = \lim_{c\rightarrow a} \frac{\kappa_1(c)f(a)+\kappa_0(c)\frac{f(c)-f(a)}{c-a}}{\kappa_1(c)g(a)+\kappa_0(c)\frac{g(c)-g(a)}{c-a}} = \frac{\kappa_1(a)f(a)+\kappa_0(a)f'(a)}{\kappa_1(a)g(a)+\kappa_0(a)g'(a)} = \frac{D^{\alpha}f(a)}{D^{\alpha}g(a)}. \]
This ends the proof.
\end{proof}

% Definition: Oscillatory %

\begin{definition}\label{oscdef}
Let $L$ be as in \eqref{saveq} with $p,q$ continuous and $p>0$, and consider the interval $[a,b)\subset(0,\infty)$ with $b\le\infty$. Then $Lx=0$ is oscillatory on $[a,b)$ provided every nontrivial real-valued solution of $Lx=0$ has infinitely many zeros in $[a,b)$, and $Lx=0$ is nonoscillatory if it is not oscillatory.
	\index{oscillatory} 
	\index{zero}
\end{definition}

The following corollary to the Sturm Comparison Theorem, Theorem \ref{sct}, relates certain behaviors of solutions to \eqref{piconeq1} and \eqref{piconeq2} in terms of disconjugacy (Definition \ref{disconj}) and oscillation (Definition \ref{oscdef}). The proof is omitted.

% Corollary to Sturm Comparison Theorem %

\begin{corollary}\label{sctcoro}
Let $p_j$ and $q_j$ be continuous functions on the interval $\mathcal{I}\subseteq\R$ with $p_j(t)>0$ for all $t\in\mathcal{I}$, for $j=1,2$. Assume that
\begin{equation}\label{kp524}
 q_2(t) \ge q_1(t), \quad p_1(t) \ge p_2(t) > 0, \quad t\in J\subset\mathcal{I}.
\end{equation}
If \eqref{piconeq2} is disconjugate on $J$, then \eqref{piconeq1} is disconjugate on $J$. If \eqref{piconeq1} is oscillatory on $J$, then \eqref{piconeq2} is oscillatory on $J$.
\end{corollary}

% Example of Sturm Comparison Corollary %

\begin{example}
We show that if
\[ 0<p(t)\le 1, \quad\text{and}\quad q(t)\ge \omega^2 > 0, \quad t\in\R, \]
then the homogeneous self-adjoint equation \eqref{saveq} is oscillatory on $\R$. To prove this we will refer to Corollary \ref{sctcoro}. Assume $\kappa_0,\kappa_1$ satisfy \eqref{kappaconditions} such that 
\[ \lim_{b\rightarrow\infty}\int_{t_0}^b 1d_{\alpha}t=\infty. \]

In \eqref{lineq} choose the constant coefficients $a\equiv 1$, $b\equiv 0$, and $c\equiv \omega^2$ for constant $\omega>0$, and $g\equiv 0$. By Lemma \ref{expprops} and Theorem \ref{pisaf} we have
\[ p_1(t)\equiv 1, \quad q_1(t) = \omega^2, \]
and the homogeneous self-adjoint equation \eqref{piconeq1} simplifies, as shown in Example \ref{coscosh}, to
\[ D^\alpha  D^\alpha x(t) + \omega^2 x(t) = 0, \quad t\in\R, \]
with oscillatory solutions of the form
\[ x(t) = c_1\cos_{\alpha}(\omega;t,t_0) + c_2\sin_{\alpha}(\omega;t,t_0), \quad t,t_0\in\R, \]
since $\omega>0$. Therefore, by Corollary \ref{sctcoro}, the homogeneous self-adjoint equation \eqref{saveq} is oscillatory on $\R$. $\hfill\triangle$
\end{example}

%%%%%%%%%%%%%%%%%%%%%%%%%%%%%
% SECTION 8 REID ROUNDABOUT %
%%%%%%%%%%%%%%%%%%%%%%%%%%%%%

\section{Reid Roundabout Theorem}\label{secreid}

The earlier results of the paper culminate here with the following roundabout theorem; this roundabout theorem generalizes and extends results found in the classical case by Reid~\cite[Theorem 2.1]{r56}, which can be recovered by setting $\alpha=1$.

%%%%%%%%%%%
% Theorem %
%%%%%%%%%%%

\begin{theorem}[Reid Roundabout Theorem]\label{reidthm}
Let $L$ be as in \eqref{saeq}, and let $p>0$ and $q$ be continuous functions on $[a,b]\subset\mathcal{I}$. Then the following are equivalent:
\begin{enumerate}
 \item[$(i)$]   There is a solution $x$ of the self-adjoint equation $Lx=0$ with $x\neq 0$ on $[a,b]$.
 \item[$(ii)$]  There is a solution $z$ of the Riccati equation \eqref{ricceq} on $[a,b]$.
 \item[$(iii)$] The quadratic functional $\cF$ given in \eqref{quadF} is positive definite on $\A$.
 \item[$(iv)$]  The self-adjoint equation $Lx=0$ is disconjugate on $[a,b]$.
 \item[$(v)$]   The solution $u$ of $Lx=0$ satisfying initial conditions $u(a)=0$ and $D^{\alpha} u(a)=1/p(a)$ is nonzero on $(a,b]$. 
 \item[$(vi)$]  The solution $v$ of $Lx=0$ satisfying initial conditions $v(b)=0$ and $D^{\alpha} v(b)=1/p(b)$ is nonzero on $[a,b)$.
\end{enumerate}
\end{theorem}

\begin{proof}
Note that $(i)$ $\Longleftrightarrow$ $(ii)$ by Theorem~\ref{saricc}.

\noindent $(ii)$ $\Longrightarrow$ $(iii)$: 
From Theorem \ref{saricc} there exists a solution $x$ of $Lx=0$ that is nonzero on $[a,b]$. Let $z$ be given by \eqref{zform}; by Theorem~\ref{saricc}, $z$ is a solution of \eqref{ricceq} on $[a,b]$. By Picone Identity I (Theorem~\ref{picone}), for any $\eta\in\A$ on $[a,b]$ we have
\begin{equation}\label{deze}
 D^{\alpha}[z\eta^2] + \kappa_1z\eta^2 = p(D^{\alpha}\eta)^2 - q\eta^2 - p(D^{\alpha}\eta - z\eta/p)^2; 
\end{equation}
using \eqref{dfpluskf}, we have that \eqref{deze} becomes
\[ e_0(t,b) D^{\alpha}\left[\frac{z\eta^2}{e_0(t,b)}\right] = p(D^{\alpha}\eta)^2 - q\eta^2 - p(D^{\alpha}\eta - z\eta/p)^2. \]
As the equation holds on $[a,b]$, we may multiply both sides by $e^2_0(b,t)$ and integrate from $a$ to $b$ to obtain
\[ \int_a^b D^{\alpha}\left[\frac{z\eta^2}{e_0(t,b)}\right]e_0(b,t)d_{\alpha}t = \cF(\eta) -\int_a^b p(D^{\alpha}\eta - z\eta/p)^2(t)e^2_0(b,t)d_{\alpha}t. \]
Since $\eta$ is admissible, we have by Theorem \ref{ftc} that
\[ \int_a^b D^{\alpha}\left[\frac{z\eta^2}{e_0(t,b)}\right]e_0(b,t)d_{\alpha}t = \frac{z\eta^2(t)}{e_0(t,b)}e_0(b,t)\Big|_{t=a}^b = 0, \]
putting
\[ \cF(\eta) = \int_a^b p(D^{\alpha}\eta - z\eta/p)^2(t)e^2_0(b,t)d_{\alpha}t. \]
As $p>0$, we have $\cF(\eta)\ge 0$ for all $\eta\in\A$. Furthermore, if $\eta\equiv 0$, then $\cF(\eta)=0$. If $\cF(\eta)=0$ for some $\eta\in\A$, then on $[a,b]$ we have
\[ D^{\alpha}\eta = z\eta/p. \]
By variation of constants \cite[Lemma 1.10]{au} we have that the initial value problem
\[ D^{\alpha}\eta - (z/p)\eta = 0, \qquad \eta(a)=0 \]
has a unique solution on $[a,b]$, namely $\eta\equiv 0$ on $[a,b]$. It follows that $\cF$ is positive definite on $\A$.

\noindent $(iii)$ $\Longrightarrow$ $(iv)$: 
Assume $Lx=0$ is not disconjugate on $[a,b]$. Then there exists a nontrivial solution $u$ of \eqref{saveq} with (at least) two zeros in $[a,b]$; in other words, there exist points $c,d\in[a,b]$ with $a\le c < d \le b$ such that $u(c)=0=u(d)$. Define
\[ \eta(t):=\begin{cases} u(t) & : t\in(c,d) \\ 0 & : \text{ otherwise.} \end{cases} \]
By Lemma~\ref{lemma511}, $\eta\in\A$ and $\cF(\eta)=0$, but $\eta\neq 0$, making $\cF$ not positive definite. By the contrapositive argument, the assertion holds.

\noindent $(iv)$ $\Longrightarrow$ $(v)$: 
Let $u$ be the solution of \eqref{saeq} satisfying 
\[ Lu=0, \quad u(a)=0, \quad D^{\alpha}u(a)=1/p(a), \]
and assume the equation \eqref{saveq} is disconjugate on $[a,b]$. Since $u(a)=0$, disconjugacy implies $u\neq 0$ on $(a,b]$.

\noindent $(iv)$ $\Longrightarrow$ $(vi)$: 
Assume \eqref{saveq} is disconjugate on $[a,b]$. The initial value problem 
\[ Lv=0, \quad v(b) = 0 \quad D^{\alpha}v(b)=1/p(b) \] 
has a unique solution $v$ by Theorem~\ref{unique}. Again $v(b)=0$, so disconjugacy implies $v\neq 0$ on $[a,b)$.

\noindent $(v)$ $\Longleftrightarrow$ $(vi)$: Let $u$ and $v$ be the unique solutions of $Lx=0$ as described in statements $(v)$ and $(vi)$, respectively. By Abel's formula, Corollary \ref{wronskid}, their Wronskian satisfies 
\[ W(u,v)(t) = e^2_0(t,a)p(a)W(u,v)(a)/p(t) = -e^2_0(t,a)v(a)/p(t); \]
 by employing the initial conditions at $a$ and $b$ we also get
\begin{equation}\label{wuv}
  -e^2_0(b,a)v(a) = u(b). 
\end{equation}
If $(v)$ holds, $u(b)\neq 0$, making $v(a)\neq 0$ also by \eqref{wuv}. For any $\tau\in(a,b]$, $u\neq 0$ on $[\tau,b]$ implies $(i)$ of this theorem holds on $[\tau,b]$, and thus $(ii)$, $(iii)$, and $(iv)$ likewise hold, as shown above. Consequently, $v$ has no second zero on $[\tau,b)$. By the arbitrary choice of $\tau$ and the condition of $v$ at $a$, $(vi)$ holds overall. The fact that $(vi)$ implies $(v)$ follows from a similar argument and is omitted.

\noindent $(vi)$ $\Longrightarrow$ $(i)$: 
As we are assuming $(vi)$, condition $(v)$ holds as well. Using \eqref{wuv}, let 
\[ y(t):= v(t)/v(a) = -e^2_0(b,a)v(t)/u(b). \]
Then $y$ is a solution of $Lx=0$ with $W(u,y)(t)=-e^2_0(t,a)/p(t)$, and $y\neq 0$ on $[a,b)$ follows from the assumptions on $v$.
Set $x:=u+y$; then $x$ is a solution of $Lx=0$ by linearity. We must check that $x$ is nonzero on all of $[a,b]$. At $b$, $y(b)=0$, and
\begin{equation}\label{xpxatb}
 x(b) = u(b) \neq 0.
\end{equation} 

Finally, let $\tau\in[a,b)$. Since $y$ is a solution of $Lx=0$ that is nonzero on $[a,\tau]$, by the second reduction of order theorem, Theorem \ref{order2}, and the definition of $y$ in terms of $v$ we have on $[a,\tau]$ that
\[ x(t) = v(t)/v(a) + v(a)v(t)\int_{a}^t \frac{e^2_{0}(s,a)}{p(s)v^2(s)} d_{\alpha}s \]
for $t\in[a,\tau]$, which we write as
\[ x(t) = \frac{v(t)}{v(a)}\left[1 + v^2(a)\int_{a}^t \frac{e^2_{0}(s,a)}{p(s)v^2(s)} d_{\alpha}s\right]. \]
Then $x\neq 0$ on $[a,\tau]$ as $v\neq 0$ on $[a,\tau]$. Given that $\tau\in[a,b)$ was arbitrary and the earlier remark \eqref{xpxatb} about $x(b)\neq 0$, we have that $x(t)\neq 0$ for all $t\in[a,b]$, and the proof is complete.
\end{proof}

%%%%%%%%%%%%%%%%%%%%%%%%%%%%%%%%%%%%%%%%%%%%%%%%%%%%%%%%%%%%%%%%%%%%%%%%%
\section{Lyapunov Inequality}\label{s103} 
%%%%%%%%%%%%%%%%%%%%%%%%%%%%%%%%%%%%%%%%%%%%%%%%%%%%%%%%%%%%%%%%%%%%%%%%%

Lyapunov inequalities have proven to be useful tools in oscillation theory, disconjugacy, eigenvalue problems, and numerous other applications in the theory of differential equations. In this section we present a conformable Lyapunov inequality.

Throughout this section we assume $a,b\in\mathcal{I}$ with $0\le a<b$ and $\mathcal{I}\subseteq[0,\infty)$. Let $q:\mathcal{I}\to\R$ be continuous with $q(t)>0$ for all $t\in\mathcal{I}$,
and consider the conformable self-adjoint equation ($p\equiv 1$)
	\index{self-adjoint equation}
\begin{equation}\label{lySL}
  D^{\alpha}D^{\alpha}x(t)+q(t)x(t) = 0.
\end{equation}
We will consider this together with the quadratic functional 
\begin{equation}\label{shortquadF}
 \cF(x) = \int_a^b \left[\left(D^{\alpha}x\right)^2-qx^2\right](t) e^2_0(b,t)d_{\alpha}t, \quad d_{\alpha}t\equiv \frac{dt}{\kappa_0(t)}. 
\end{equation}
To prove a Lyapunov inequality for \eqref{lySL} we will first need several auxiliary results.

%%%%%%%%%%%%%%%%%%%%%%%%%%%%%%%%%
% Lemma 1 Lyapunov's Inequality %
%%%%%%%%%%%%%%%%%%%%%%%%%%%%%%%%%

\begin{lemma}\label{lyl31}
If $x$ solves \eqref{lySL} and if $\cF(y)$ is defined for a function $y$, then
\[ \cF(y)-\cF(x) = \cF(y-x) + 2[(y-x)D^{\alpha}x](t)e^2_0(b,t)\big|_{t=a}^{b}. \]
\end{lemma}

\begin{proof}
Under the above assumptions we find
\begin{eqnarray*}
\lefteqn{\cF(y)-\cF(x)-\cF(y-x) = \int_a^b \left[(D^{\alpha}y)^2-qy^2-(D^{\alpha}x)^2+qx^2\right.} \\
&&  \qquad\qquad\qquad
    \left. -(D^{\alpha}(y-x))^2+q(y-x)^2\right](t)e^2_0(b,t)d_{\alpha}t \\
&=& 2\int_a^b \left[(D^{\alpha}y)(D^{\alpha}x) - qyx + qx^2 - (D^{\alpha}x)^2\right](t)e^2_0(b,t)d_{\alpha}t \\
&\stackrel{\eqref{lySL}}{=}& 2\int_a^b \left[(D^{\alpha}x)(D^{\alpha}(y-x)) + (y-x) D^{\alpha}[D^{\alpha}x]\right](t)e^2_0(b,t)d_{\alpha}t \\
&=& 2\int_a^b \left\{D^{\alpha}\left[(y-x)D^{\alpha}x\right] + \kappa_1(y-x)D^{\alpha}x\right\}(t)e^2_0(b,t)d_{\alpha}t \\
&=& 2\int_a^b D^{\alpha}\left[\frac{(y-x)D^{\alpha}x}{e_0(t,b)}\right](t) e_0(t,b)e^2_0(b,t)d_{\alpha}t \\
&=& 2[(y-x)D^{\alpha}x](t)e^2_0(b,t)\big|_{t=a}^{b},
\end{eqnarray*}
using Theorem \ref{ftc} and Lemma \ref{basicderiv} (iii).
\end{proof}

%%%%%%%%%%%%%%%%%%%%%%%%%%%%%%%%%
% Lemma 2 Lyapunov's Inequality %
%%%%%%%%%%%%%%%%%%%%%%%%%%%%%%%%%

\begin{lemma}\label{lyl32}
Let $\cF$ be given by \eqref{shortquadF}. If $\cF(y)$ is defined, then for any $r,s\in\mathcal{I}$ with $a \le r < s \le b$,
\[ \int_r^s [D^{\alpha}y]^2(t) e^2_0(s,t)d_{\alpha}t \ge \frac{\left\{y(s) - y(r)e_0(s,r)\right\}^2}{h_1(s,r)}. \]
\end{lemma}

\begin{proof}
Under the above assumptions we define
\[ x(t):= \frac{y(s)e_0(t,s)h_1(t,r) - y(r)e_0(t,r)h_1(t,s)}{h_1(s,r)}, \quad h_1(b,a)=\int_a^b 1d_{\alpha}t. \]
We then have
\[ x(r)=y(r),\quad x(s)=y(s),\quad D^{\alpha}x(t) = \frac{y(s)e_0(t,s) - y(r)e_0(t,r)}{h_1(s,r)}, \]
and 
\[ D^{\alpha}D^{\alpha}x(t) = 0, \]
where we used Lemma \ref{basicderiv} (ii), (iv), and (v).
Hence $x$ solves the special conformable equation \eqref{lySL} where $q=0$, and therefore we may apply Lemma \ref{lyl31} to  $\cF_0$ defined by
\[ \cF_0(x):= \int_r^s \left[D^{\alpha}x\right]^2(t) e^2_0(s,t)d_{\alpha}t \]
to find
\begin{eqnarray*}
\cF_0(y)
&=& \cF_0(x) + \cF_0(y-x) + 2[(y-x)D^{\alpha}x](t)e^2_0(s,t)\big|_{t=r}^{s} \\
&=& \cF_0(x)+\cF_0(y-x) \\
&\ge& \cF_0(x) \\ 
&=& \int_r^s \left[D^{\alpha}x\right]^2(t) e^2_0(s,t)d_{\alpha}t  \\
&=& \int_r^s \left[\frac{y(s)e_0(t,s) - y(r)e_0(t,r)}{h_1(s,r)}\right]^2 e^2_0(s,t)d_{\alpha}t \\
&=& \int_r^s \left[\frac{y(s) - y(r)e_0(s,r)}{h_1(s,r)}\right]^2 d_{\alpha}t \\
&=& \left[\frac{y(s) - y(r)e_0(s,r)}{h_1(s,r)}\right]^2 h_1(s,r),
\end{eqnarray*}
and this proves our claim.
\end{proof}

Using Lemma \ref{lyl32}, we now can prove a Lyapunov inequality for conformable equations of the form \eqref{lySL}.

%%%%%%%%%%%%%%%%%%%%%%%%%%%%%%%%%%
% Theorem: Lyapunov's Inequality %
%%%%%%%%%%%%%%%%%%%%%%%%%%%%%%%%%%

\begin{theorem}[Lyapunov's Inequality]\label{lyt31}
	\index{Lyapunov's inequality!scalar case}
Let $q:\mathcal{I}\to(0,\infty)$ be continuous and positive-valued. If the conformable equation \eqref{lySL} has a nontrivial solution $x$ with $x(a)=x(b)=0$, then the Lyapunov inequality
\begin{equation}\label{lye333}
  \int_a^b q(t) e_0(b,t)d_{\alpha}t \ge \frac{4e_0(b,a)}{h_1(b,a)}, 
\end{equation}
holds.
\end{theorem}

\begin{proof}
Suppose $x$ is a nontrivial solution of \eqref{lySL} with $x(a)=x(b)=0$. But then we have from Lemma \ref{lyl31} (with $y=0$) that 
\[ \cF(x) = \int_a^b\left[(D^{\alpha}x)^2-qx^2\right](t) e^2_0(b,t)d_{\alpha}t = 0. \]
Since $x$ is nontrivial, we have that $M$ defined by 
\begin{equation}\label{lyeM}
	M:=\max\left\{x^2(t):\;t\in[a,b]\right\} 
\end{equation}
is positive. We now let $c\in(a,b)$ be such that $x^2(c)=M$. According to Definition \ref{maxmin}, we then have
\[ x^2(c)\ge e_0(c,t)x^2(t), \quad t\in[a,b]. \]
Applying the above as well as Lemma \ref{lyl32} twice (once with $r=a$ and $s=c$ and a second time with $r=c$ and $s=b$) we find
\begin{eqnarray*}
 M \int_a^b q(t) e_0(b,t)d_{\alpha}t 
  &\ge& \int_a^b \left\{qx^2\right\}(t) e_0(c,t)e_0(b,t)d_{\alpha}t \\
  & = & e_0(c,b) \int_a^b q(t)x^2(t)e^2_0(b,t) d_{\alpha}t \\
  & = & e_0(c,b) \int_a^b (D^{\alpha}x)^2(t) e^2_0(b,t)d_{\alpha}t \\
  & = & e_0(c,b)\left[\int_a^c (D^{\alpha}x)^2(t) e^2_0(b,t)d_{\alpha}t + \int_c^b (D^{\alpha}x)^2(t) e^2_0(b,t)d_{\alpha}t\right] \\
  & = & e_0(b,c)\int_a^c (D^{\alpha}x)^2(t) e^2_0(c,t)d_{\alpha}t + e_0(c,b)\int_c^b (D^{\alpha}x)^2(t) e^2_0(b,t)d_{\alpha}t \\
  &\ge& \frac{e_0(b,c)}{h_1(c,a)}\left\{x(c) - x(a)e_0(c,a)\right\}^2 + \frac{e_0(c,b)}{h_1(b,c)}\left\{x(b) - x(c)e_0(b,c)\right\}^2 \\
  & = & x^2(c) \left\{\frac{e_0(b,c)}{h_1(c,a)} + \frac{e_0(c,b)e^2_0(b,c)}{h_1(b,c)}\right\} \\
  & = & Me_0(b,c) \left\{\frac{1}{h_1(c,a)} + \frac{1}{h_1(b,c)}\right\} \\
  &\ge& Me_0(b,c)e_0(c,c^*) \left\{\frac{2+\sqrt{4+\kappa_1^2(c^*)h_1(b,a)}}{h_1(b,a)}\right\} \\
  &\ge& M \left\{\frac{4e_0(b,a)}{h_1(b,a)}\right\}.
\end{eqnarray*}
If we divide by $M>0$ we arrive at the desired inequality. Here we have used, in the penultimate inequality, the fact that $h^{-1}_1(c,a)+h^{-1}_1(b,c)$ reaches an $\alpha$-minimum at $c^*\in(a,b)$ iff
\begin{equation}\label{lyapcstar}
 D^{\alpha}\left[\frac{1}{h_1(c,a)} + \frac{1}{h_1(b,c)}\right] = 0 \quad\text{at $c^*\in(a,b)$ iff}\quad \frac{1}{h_1(c^*,a)} = \frac{1}{h_1(b,c^*)} +\kappa_1(c^*), 
\end{equation}
and the fact that $h_1(c^*,a)=h_1(b,a)-h_1(b,c^*)$.
\end{proof}

As an application of Theorem \ref{lyt31} we now prove a sufficient criterion for disconjugacy of \eqref{lySL}, for disconjugacy defined in Definition \ref{disconj}. We use the notation of Theorem \ref{lyt31}.

% Theorem %

\begin{theorem}[Sufficient Condition for Disconjugacy of \eqref{lySL}]\label{lyt32}
	\index{disconjugate}
If $q$ satisfies 
\begin{equation}\label{lye33}
	\int_a^b q(t)e_0(b,t) d_{\alpha}t < \frac{4e_0(b,a)}{h_1(b,a)}, 
\end{equation}
then \eqref{lySL} is disconjugate on $[a,b]$.
\end{theorem}

\begin{proof}
Suppose that \eqref{lye33} holds. For the sake of contradiction we assume \eqref{lySL} is not disconjugate. But then by the Reid Roundabout Theorem \ref{reidthm}, there exists a nontrivial $x$ with $x(a)=x(b)=0$ such that $\cF(x)\le 0$. Using this $x$, we define $M$ by \eqref{lyeM} to find
\begin{eqnarray*}
 M \int_a^b q(t) e_0(b,t)d_{\alpha}t 
  &\ge& \int_a^b \left\{qx^2\right\}(t) e_0(c,t)e_0(b,t)d_{\alpha}t \\
  & = & e_0(c,b) \int_a^b (D^{\alpha}x)^2(t) e^2_0(b,t)d_{\alpha}t \\
  & = & e_0(b,c)\int_a^c (D^{\alpha}x)^2(t) e^2_0(c,t)d_{\alpha}t + e_0(c,b)\int_c^b (D^{\alpha}x)^2(t) e^2_0(b,t)d_{\alpha}t \\
  &\ge& \frac{e_0(b,c)}{h_1(c,a)}\left\{x(c) - x(a)e_0(c,a)\right\}^2 + \frac{e_0(c,b)}{h_1(b,c)}\left\{x(b) - x(c)e_0(b,c)\right\}^2 \\
  & = & Me_0(b,c) \left\{\frac{1}{h_1(c,a)} + \frac{1}{h_1(b,c)}\right\} \\
  &\ge& Me_0(b,c)e_0(c,c^*) \left\{\frac{2+\sqrt{4+\kappa_1^2(c^*)h_1(b,a)}}{h_1(b,a)}\right\} \\
  &\ge& M \left\{\frac{4e_0(b,a)}{h_1(b,a)}\right\}.
\end{eqnarray*}
where the last inequality follows precisely as in the proof of Theorem \ref{lyt31}. Consequently, after dividing by $M>0$, we arrive at
\[ \int_a^b q(t)e_0(b,t) d_{\alpha}t \ge \frac{4e_0(b,a)}{h_1(b,a)}, \]
which contradicts \eqref{lye33}, and thus completes the proof.
\end{proof}

% Example %

\begin{example}
Theorem \ref{lyt32} gives a sufficient condition for the disconjugacy of \eqref{lySL}, namely if the inequality
\eqref{lye333}, namely
\[ \int_a^b q(t) e_0(b,t)d_{\alpha}t < \frac{4e_0(b,a)}{h_1(b,a)}, \]
holds. This example shows that this result is sharp in the sense that 4 is the largest constant such that this result is true in general. To this end, on $[a,b]=[0,1]$, pick $c,\delta\in(0,1)$ such that $h_1(c,0)=h_1(1,c)$ and $c\pm\delta\in(0,1)$. Choose $x:[0,1]\rightarrow\R$ so that $x$ has a continuous second derivative on $[0,1]$ with
\[ x(t) = e_0(t,c-\delta)\begin{cases} h_1(t,0) &: t\in[0,c-\delta], \\ h_1(1,t) &: t\in[c+\delta,1], \end{cases} \]
and such that $x$ satisfies 
\[ D^{\alpha}D^{\alpha}x(t) \le 0, \quad t\in[0,1]. \]
Moreover, let
\[ q(t):= \begin{cases} \frac{-D^{\alpha}D^{\alpha}x(t)}{x(t)}e_0(t,0) &: t\in(0,1), \\ 0 &: t=0,1. \end{cases} \]
Note that $q$ is nonnegative and continuous on $[0,1]$, and $x$ is a nontrivial solution of \eqref{lySL}, namely
\[ D^{\alpha}D^{\alpha}x(t)+q(t)x(t)=0 \]
with $x(0)=0=x(1)$; consequently, \eqref{lySL} is not disconjugate on $[0,1]$. Additionally, we see that
\begin{eqnarray*}
 \int_0^1 q(t) e_0(1,t)d_{\alpha}t 
 & = & \int_{c-\delta}^{c+\delta} \frac{-D^{\alpha}D^{\alpha}x(t)}{x(t)}e_0(t,0)e_0(1,t)d_{\alpha}t \\
 &\le& \frac{-1}{x(c-\delta)} \int_{c-\delta}^{c+\delta} D^{\alpha}[D^{\alpha}x](t)e_0(c-\delta,t)e_0(1,0)d_{\alpha}t \\
 &\le& \frac{-e_0(c-\delta,0)e_0(1,c+\delta)}{h_1(c-\delta,0)} \int_{c-\delta}^{c+\delta} D^{\alpha}[D^{\alpha}x](t)e_0(c+\delta,t)d_{\alpha}t \\
 &\stackrel{\text{Thm } \ref{ftc}}{=}& \frac{-e_0(c-\delta,0)e_0(1,c+\delta)}{h_1(c-\delta,0)} \left[D^{\alpha}x(c+\delta) - D^{\alpha}x(c-\delta)e_0(c+\delta,c-\delta)\right] \\
 & = & \frac{-e_0(c-\delta,0)e_0(1,c+\delta)}{h_1(c-\delta,0)} \left[-e_0(c+\delta,c-\delta) - e_0(c+\delta,c-\delta)\right] \\
 & = & \frac{2e_0(1,0)}{h_1(c-\delta,0)} \\
 & = & \frac{4e_0(1,0)}{h_1(1,0)-2h_1(c,c-\delta)}
\end{eqnarray*}
using the fact that $h_1(c,0)=h_1(1,c)$ and $h_1(c,0)+h_1(1,c)=h_1(1,0)$. Comparing with the result in \eqref{lye333}, clearly
\[ \frac{4e_0(1,0)}{h_1(1,0)-2h_1(c,c-\delta)} > \frac{4e_0(1,0)}{h_1(1,0)} \]
but they can be made arbitrarily close by taking $\delta$ arbitrarily close to 0. 
\end{example}

%%%%%%%%%%%%%%%%%%%%%%%%%%%%
% SECTION 10 Factorization %
%%%%%%%%%%%%%%%%%%%%%%%%%%%%

\section{Factorization}\label{fact}
In this section we lay the groundwork for further exploration of the nonhomogeneous equation \eqref{saeq} by introducing the P\'{o}lya factorization for the conformable self-adjoint equation $Lx=0$, which in turn leads to a variation of parameters result for $Lx=h$. Again we assume throughout that the coefficient function $p$ satisfies $p>0$.

%%%%%%%%%%%%%%%%%%%%%%%
% Polya Factorization %
%%%%%%%%%%%%%%%%%%%%%%%

\begin{theorem}[P\'{o}lya Factorization]\label{polya}
Let $L$ be as in \eqref{saeq}. If $Lx=0$ has a positive solution $x>0$ on an interval $[a,b)\subset(0,\infty)$ with $b\le\infty$, then for any function $y\in\D$ we have a P\'{o}lya factorization of $Ly$ on $[a,b)$ given by 
\[ Ly = \rho_1D^{\alpha}\left\{\rho_2D^{\alpha}[\rho_1 y]\right\}, \quad \rho_1(t):=\frac{e_0(t,a)}{x(t)} > 0, \quad \rho_2(t):=\frac{p(t)x^2(t)}{e^2_0(t,a)} > 0. \]
\end{theorem}

\begin{proof}
In the calculations below we will use the following conformable derivative facts, namely that
\[ D^{\alpha} f(t) + \kappa_1(t) f(t) \stackrel{\eqref{dfpluskf}}{=} e_0(t,a)D^{\alpha}\left[\frac{f}{e_0(\cdot,a)}\right](t) \]
and
\[ D^{\alpha} f(t) - \kappa_1(t) f(t) \stackrel{\eqref{dfminuskf}}{=} \frac{1}{e_0(t,a)}D^{\alpha}\left[e_0(\cdot,a)f\right](t). \]
Assume $x>0$ is a positive solution of $Lx=0$ on $[a,b)$, and let $y\in\D$. Then
\begin{eqnarray*}
 Ly &\stackrel{\text{Thm }\ref{wnab}}=& \frac{1}{x}\left\{D^{\alpha}[pW(x,y)]+\kappa_1pW(x,y)\right\} \\
    &\stackrel{\eqref{dfpluskf}}=& \frac{e_0(\cdot,a)}{x}D^{\alpha}\left[\frac{pW(x,y)}{e_0(\cdot,a)}\right] \\
    &\stackrel{\text{Def }\ref{wdef}}=& \frac{e_0(\cdot,a)}{x}D^{\alpha}\left\{\frac{p}{e_0(\cdot,a)}\left[xD^{\alpha}y - yD^{\alpha}x\right]\right\} \\
    &=& \rho_1D^{\alpha}\left\{\frac{px^2}{e_0(\cdot,a)}\left[\frac{1}{x}D^{\alpha}y - \frac{y}{x^2}D^{\alpha}x\right]\right\} \\
    &\stackrel{\eqref{dxinvu}}{=}& \rho_1D^{\alpha}\left\{e_0(\cdot,a)\rho_2\left[D^{\alpha}[y/x]-\kappa_1y/x\right]\right\} \\
    &\stackrel{\eqref{dfminuskf}}{=}& \rho_1D^{\alpha}\left\{e_0(\cdot,a)\rho_2\frac{1}{e_0(\cdot,a)}D^{\alpha}\left[e_0(\cdot,a)y/x\right]\right\} \\
    &=& \rho_1D^{\alpha}\left\{\rho_2D^{\alpha}\left[\rho_1y\right]\right\},
\end{eqnarray*}
for $\rho_1$ and $\rho_2$ as defined in the statement of the theorem.
\end{proof}

%%%%%%%%%%%%%%%%%%
% Theorem V of P %
%%%%%%%%%%%%%%%%%%

\begin{theorem}[Variation of Parameters]\label{variation}
Let $h$ be a continuous function defined on $[a,\infty)$, and let $L$ be given as in \eqref{saeq}. If the homogeneous equation $Lx=0$ has a positive solution $x>0$ on $[a,\infty)$, then the nonhomogeneous equation $Ly=h$ has a solution $y$ given by
\begin{eqnarray*}
  y(t) &=& \frac{x(t)y(a)}{x(a)} + p(a)W(x,y)(a)x(t)\int_a^t \frac{e^2_0(s,a)}{p(s)x^2(s)}d_{\alpha}s \\    
       & & + x(t)\int_a^t \left(\frac{1}{p(s)x^2(s)}\int_a^s x(\sigma)h(\sigma)e^2_0(s,\sigma)d_{\alpha}\sigma\right)d_{\alpha}s. 
\end{eqnarray*}
\end{theorem}

\begin{proof}
Let $y\in\D$ be defined on $[a,\infty)$, and assume $x>0$ is a positive solution of $Lx=0$ on $[a,\infty)$. As in Theorem \ref{polya}, we factor $Ly$ to get
\[ h(\sigma) = Ly(\sigma) = \frac{e_0(\sigma,a)}{x(\sigma)}D^{\alpha}\left\{\frac{px^2}{e^2_0(\cdot,a)}D^{\alpha}[e_0(\cdot,a)y/x]\right\}(\sigma). \]
Multiplying by $e_0(t,\sigma)e_0(a,\sigma)x$ and integrating from $a$ to $t$ we arrive via Theorem \ref{ftc} at
\begin{eqnarray*} 
 && \int_a^t e_0(a,\sigma)x(\sigma)h(\sigma)e_0(t,\sigma)d_{\alpha}\sigma \\
 && =\int_a^t D^{\alpha}\left\{\frac{px^2}{e^2_0(\cdot,a)}D^{\alpha}[e_0(\cdot,a)y/x]\right\}(\sigma)e_0(t,\sigma)d_{\alpha}\sigma \\
 && =\frac{(px^2)(\sigma)}{e^2_0(\sigma,a)}D^{\alpha}[e_0(\cdot,a)y/x](\sigma)e_0(t,\sigma)\big|_{\sigma=a}^t \\
 && =\frac{(px^2)(t)}{e^2_0(t,a)}D^{\alpha}[e_0(\cdot,a)y/x](t)e_0(t,t) - \frac{(px^2)(a)}{e^2_0(a,a)}D^{\alpha}[e_0(\cdot,a)y/x](a)e_0(t,a) \\
 && =\frac{(px^2)(t)}{e^2_0(t,a)}D^{\alpha}[e_0(\cdot,a)y/x](t) - (px^2)(a)D^{\alpha}[e_0(\cdot,a)y/x](a)e_0(t,a) \\
 && =\frac{(px^2)(t)}{e^2_0(t,a)}D^{\alpha}[e_0(\cdot,a)y/x](t) - p(a)W(x,y)(a)e_0(t,a)
\end{eqnarray*}
using Theorem \ref{order2}, since $x>0$ is a solution. This leads to
\begin{eqnarray*} 
 D^{\alpha}[e_0(\cdot,a)y/x](s)e_0(t,s) d_{\alpha}s &=& \frac{e_0(t,a)p(a)}{(px^2)(s)}W(x,y)(a)e^2_0(s,a)d_{\alpha}s \\
 & & + \frac{e_0(t,a)}{(px^2)(s)}\left(\int_a^s x(\sigma)h(\sigma)e^2_0(s,\sigma)d_{\alpha}\sigma\right)d_{\alpha}s.
\end{eqnarray*}
Integrating this from $a$ to $t$ and using Theorem \ref{ftc} yields the form for $y$ given in the statement of the theorem. Clearly the right-hand side of the form of $y$ above reduces to $y(a)$ at $a$, and since $x>0$ is a solution the conformable derivative reduces to $D^{\alpha}y(a)$ at $a$.
\end{proof}

%%%%%%%%%%%%%
% Corollary %
%%%%%%%%%%%%%

\begin{corollary}
Let $h$ be a continuous function defined on $[a,\infty)$, and let $L$ be given as in \eqref{saeq}. If the homogeneous matrix equation \eqref{saeq} has a positive solution $x>0$ on $[a,\infty)$, then the nonhomogeneous initial value problem
\begin{equation}\label{nonivp}
 Ly = h, \quad y(a) = y_a, \quad D^{\alpha}y(a) = y_a'
\end{equation}
has a unique solution.
\end{corollary}

\begin{proof}
By Theorem \ref{variation}, the nonhomogeneous initial value problem \eqref{nonivp} has a solution. Suppose $y_1$ and $y_2$ both solve \eqref{nonivp}. Then $x=y_1-y_2$ solves the homogeneous initial value problem
\[ Lx=0, \quad x(a)=0, \quad D^{\alpha}x(a)=0; \]
by Theorem \ref{unique}, this has only the trivial solution $x\equiv 0$, and thus $y_1=y_2$ is unique.
\end{proof}

%%%%%%%%%%%%%%%%%%%%%%%%%%%%%%%%%%%%%%
% Theorem Trench factorization(3.15) %
%%%%%%%%%%%%%%%%%%%%%%%%%%%%%%%%%%%%%%

\begin{theorem}[Trench Factorization]\label{trenchscalar}
	\index{Trench factorization}
Let $L$ be as in \eqref{saeq}. If $Lx=0$ has a positive solution on an interval $[a,b)\subset(0,\infty)$ with $b\le\infty$, then for any $x\in\D$ we have the Trench factorization
\[ Lx(t) = \gamma_1(t)D^{\alpha}\left\{\gamma_2 D^{\alpha}[\gamma_1x]\right\}(t) \quad\text{for all}\quad t\in[a,b), \]
where $\gamma_1$ and $\gamma_2$ are positive functions on $[a,b)$ and
\begin{equation}\label{eq32}
   \int_a^{b}\frac{1}{\gamma_2(t)}d_{\alpha}t = \infty.
\end{equation}
\end{theorem}

\begin{proof} 
Since $Lx=0$ has a positive solution on $[a,b)$, $Lx$ has a Polya factorization on $[a,b)$ by Theorem \ref{polya}.
Hence there are positive functions $\rho_1$ and $\rho_2$ on $[a,b)$ such that
\[ Lx(t) = \rho_1(t)D^{\alpha}\left\{\rho_2D^{\alpha}[\rho_1x]\right\}(t)
   = \frac{1}{\delta_1(t)}D^{\alpha}\left\{\frac{1}{\delta_2} D^{\alpha}\left[\frac{x}{\delta_1}\right]\right\}(t) \]
for $t\in[a,b)$, where
\[ \delta_i(t):=\frac{1}{\rho_i(t)} \quad\text{for $i=1,2$,}\quad t\in[a,b). \]
If
\[ \int_a^{b}\delta_2(t)d_{\alpha}t = \infty, \]
then we have what we want. Assume
\begin{equation}\label{assumediesda}
 \int_a^{b}\delta_2(t)d_{\alpha}t < \infty. 
\end{equation}
In this case we set
\[ \beta_1(t) = \delta_1(t)\int_t^{b}\delta_2(s) d_{\alpha}s  \quad\text{and}\quad \beta_2(t) = \frac{\delta_2(t)}{\left(\int_t^{b}\delta_2(s) d_{\alpha}s\right)^2} \]
for $t\in[a,b)$. Then by \eqref{assumediesda}
\begin{eqnarray*}
  \int_a^{b}\beta_2(t)d_{\alpha}t
   &=& \lim_{c\to b^-} \int_a^c \frac{\delta_2(t)}{\left(\int_t^{b}\delta_2(s) d_{\alpha}s\right)^2}d_{\alpha}t \\
   &=& \lim_{c\to b^-} \int_a^c D^{\alpha}\left\{\frac{e_0(t,a)}{\int_t^{b}\delta_2(s) d_{\alpha}s}\right\} e_0(a,t) d_{\alpha}t \\
   &=& \lim_{c\to b^-} \frac{e_0(t,a)e_0(a,t)}{\int_t^{b}\delta_2(s) d_{\alpha}s}\Big|_{t=a}^c \\
   &=& \infty.
\end{eqnarray*}
For $x\in\D$, note that
\[ D^{\alpha}\left[\frac{x}{\beta_1}\right](t)
   = D^{\alpha}\left[\frac{\frac{x(t)}{\delta_1(t)}} {\int_t^{b}\delta_2(s) d_{\alpha}s}\right] \\
   = \frac{\left(\int_t^{b}\delta_2(s) d_{\alpha}s\right) D^{\alpha}\left[\frac{x}{\delta_1}\right](t) + \frac{x(t)\delta_2(t)}{\delta_1(t)}}
      {\left(\int_t^{b}\delta_2(s) d_{\alpha}s\right)^2} \]
for $t\in[a,b)$. Hence
\[ \frac{1}{\beta_2(t)}D^{\alpha}\left[\frac{x}{\beta_1}\right](t) = \frac{1}{\delta_2(t)}D^{\alpha}\left[\frac{x}{\delta_1}\right](t)
    \int_t^{b}\delta_2(s)  d_{\alpha}s + \frac{x(t)}{\delta_1(t)} \]
for $t\in[a,b)$. Applying the operator $D^{\alpha}$ to both sides we get 
\[ D^{\alpha}\left\{\frac{1}{\beta_2}D^{\alpha}\left[\frac{x}{\beta_1}\right]\right\}(t)
   = D^{\alpha}\left\{\frac{1}{\delta_2}D^{\alpha}\left[\frac{x}{\delta_1}\right]\right\}(t)\int_{t}^{b}\delta_2(s) d_{\alpha}s \]
for $t\in[a,b)$. It follows that
\[ \frac{1}{\beta_1(t)}D^{\alpha}\left\{\frac{1}{\beta_2} D^{\alpha}\left[\frac{x}{\beta_1}\right]\right\}(t)
   = \frac{1}{\delta_1(t)}D^{\alpha}\left\{\frac{1}{\delta_2} D^{\alpha}\left[\frac{x}{\delta_1}\right]\right\}(t)
   = Lx(t) \]
for $t\in[a,b)$. If
\[ \gamma_i(t):=\frac{1}{\beta_i(t)} \quad\text{for $i=1,2$,}\quad t\in[a,b), \]
then
\[ Lx(t) = \gamma_1(t)D^{\alpha}\{\gamma_2D^{\alpha}[\gamma_1x]\}(t) \]
for $t\in[a,b)$, where \eqref{eq32} is satisfied.
\end{proof}

%%%%%%%%%%%%%%%%%%%%%%%%%%%%%%%%%%%%%
% Theorem 29 Recessive and dom sols %
%%%%%%%%%%%%%%%%%%%%%%%%%%%%%%%%%%%%%

\begin{theorem}\label{rdsols}
	\index{recessive solution}
	\index{dominant solution}
Let $L$ be as in \eqref{saeq}. If $Lx=0$ is nonoscillatory on an interval $[a,b)\subset(0,\infty)$ with $b\le\infty$, then there is a positive solution $u$, called a recessive solution at $b$, such that for any second linearly independent solution $v$, called a dominant solution at $b$,
\[ \lim_{t\to b^-}\frac{u(t)}{v(t)} = 0, \quad
   \int_{a}^{b}\frac{e^2_0(t,a)}{p(t)u^2(t)}d_{\alpha}t = \infty, \quad\text{and}\quad
   \int_{t_0}^{b}\frac{e^2_0(t,a)}{p(t)v^2(t)}d_{\alpha}t < \infty, \]
where $t_0<b$ is sufficiently close. Furthermore
\begin{equation}\label{rin}
   \frac{p(t)D^{\alpha}v(t)}{v(t)} > \frac{p(t)D^{\alpha}u(t)}{u(t)}  
\end{equation}
for $t<b$ sufficiently close. Moreover, the recessive solution is unique up to multiplication by a nonzero constant.
\end{theorem}

\begin{proof} 
Since we are assuming that $Lx=0$ is nonoscillatory on $[a,b)$, without loss of generality it has a positive solution on $[a,b)$. Thus, $Lx$ has a Trench factorization on $[a,b)$ by Theorem \ref{trenchscalar}, namely
\[ Lx(t) = \gamma_1(t)D^{\alpha}\left\{\gamma_2 D^{\alpha}[\gamma_1x]\right\}(t) \quad\text{for all}\quad t\in[a,b), \]
where $\gamma_1$ and $\gamma_2$ are positive functions on $[a,b)$, and \eqref{eq32} holds. If 
\[ u(t) = \frac{e_0(t,a)}{\gamma_1(t)} > 0, \]
then we get from the Trench factorization that $u$ is a positive solution of $Lx=0$ on $[a,b)$. Let 
\[ v_0(t) = \frac{e_0(t,a)}{\gamma_1(t)}\int_a^t\frac{1}{\gamma_2(s)} d_{\alpha}s = u(t)\int_a^t \gamma^{-1}_2(s) d_{\alpha}s > 0. \]
Since $\gamma_2(t)D^{\alpha}[\gamma_1v_0](t)=e_0(t,a)$ and $v_0(a)=0$, we get by the Trench factorization that $v_0$ is a solution of $Lx=0$ on $[a,b)$. Note that
\[ \lim_{t\to b^-}\frac{u(t)}{v_0(t)} = \lim_{t\to b^-}\frac{1}{\int_a^t \gamma^{-1}_2(s) d_{\alpha}s} = 0 \]
by \eqref{eq32}. Consider, using the quotient rule from Lemma \ref{basicderiv}, the expression
\begin{equation}\label{dvou}
 D^{\alpha}\left[\frac{v_0}{u}\right](t) -\kappa_1(t)\frac{v_0(t)}{u(t)} = \frac{W(u,v_0)(t)}{u^2(t)} = \frac{ce_0^2(t,a)}{p(t)u^2(t)},
\end{equation}
where $c=p(a)W(u,v_0)(a)$, by Abel's formula (Corollary \ref{wronskid}). Note that $c\neq 0$ since $u$ and $v_0$ are linearly independent. 
Using \eqref{dfminuskf} we see that
\[ D^{\alpha}\left[\frac{v_0}{u}\right](t) -\kappa_1(t)\frac{v_0(t)}{u(t)} = \frac{1}{e_0(t,a)} D^{\alpha}\left[e_0(\cdot,a)\frac{v_0}{u}\right](t), \]
so that
\[ \frac{1}{e_0(t,a)} D^{\alpha}\left[e_0(\cdot,a)\frac{v_0}{u}\right](t) = \frac{W(u,v_0)(t)}{u^2(t)} = \frac{ce_0^2(t,a)}{p(t)u^2(t)}. \] 
Integrating both sides of this last equality from $a$ to $t$ via Theorem \ref{ftc}, we obtain
\[ 0 < \frac{v_0(t)}{u(t)} = \int_a^t \frac{ce^2_0(s,a)}{p(s)u^2(s)} d_{\alpha}s. \]
Letting $t\to b^-$ we get one of the desired results, namely
\[ \int_a^{b} \frac{e^2_0(t,a)}{p(t)u^2(t)}d_{\alpha}t = \infty. \]
Suppose $v$ is any solution of $Lx=0$ such that $u$ and $v$ are linearly independent. Then
\[ v(t)=c_1u(t)+c_2v_0(t), \]
where $c_2\neq 0$. It follows that
\[ \lim_{t\to b^-}\frac{u(t)}{v(t)} = \lim_{t\to b^-}\frac{u(t)}{c_1u(t)+c_2v_0(t)} = \lim_{t\to b^-}\frac{\frac{u(t)}{v_0(t)}}{c_1\frac{u(t)}{v_0(t)}+c_2} = 0. \]
Again let $v$ be a fixed solution of \eqref{saeq} such that $u$ and $v$ are linearly independent, but this time pick $t_0\in[a,b)$ so that $v(t)\neq 0$ on $[t_0,b)$, which is possible by the nonoscillatory nature of $Lx=0$. Then for $t\in[t_0,b)$, similar to the calculations above in \eqref{dvou}, we have
\begin{eqnarray*}
  D^{\alpha}\left[\frac{u}{v}\right](t) -\kappa_1(t)\frac{u(t)}{v(t)} = 
  \frac{1}{e_0(t,a)} D^{\alpha}\left[e_0(\cdot,a)\frac{u}{v}\right](t) = \frac{W(v,u)(t)}{v^2(t)} = \frac{c_2e_0^2(t,a)}{p(t)v^2(t)},
\end{eqnarray*}
where $c_2\neq 0$. Integrating both sides of this last equality from $t_0$ to $t$ we obtain
\[ \frac{u(t)}{v(t)}-\frac{u(t_0)}{v(t_0)} = \int_{t_0}^t\frac{c_2e_0^2(s,a)}{p(s)v^2(s)} d_{\alpha}s. \]
Letting $t\to b^-$ we get another one of the desired results, that is
\[ \int_{t_0}^{b}\frac{e_0^2(t,a)}{p(t)v^2(t)}d_{\alpha}t < \infty. \]
We now show that \eqref{rin} holds for $t<b$ sufficiently close. Above we saw that $v(t)\neq 0$ on $[t_0,b)$. Note that the expression
\[ \frac{p(t)D^{\alpha}v(t)}{v(t)} \]
is the same if $v(t)$ is replaced by $-v(t)$. Hence without loss of generality we can assume $v>0$ on $[t_0,b)$. For $t\in[t_0,b)$, consider
\[ \frac{p(t)D^{\alpha}v(t)}{v(t)}-\frac{p(t)D^{\alpha}u(t)}{u(t)} = \frac{p(t)W(u,v)(t)}{v(t)u(t)}
   = \frac{c_3e_0^2(t,t_0)}{v(t)u(t)}, \]
where $c_3=p(t)W(u,v)(t)/e_0^2(t,t_0)$ is a (nonzero) constant by Abel's formula. 
It remains to show that $c_3>0$. Since
\[ \lim_{t\to b^-}\frac{v(t)}{u(t)} = \infty, \]
the ordinary derivative $(v/u)'>0$ near $b$. Consequently, we see that
\[ 0 < \kappa_0(t)\left(\frac{v}{u}\right)'(t) \stackrel{\eqref{derivdef}}{=} D^{\alpha}\left[\frac{v}{u}\right](t) - \kappa_1(t)\frac{v(t)}{u(t)} = \frac{W(u,v)(t)}{u^2(t)} = \frac{c_3e_0^2(t,t_0)}{p(t)u^2(t)}, \]
and we get the desired result that $c_3>0$, and this completes the proof.
\end{proof}

% Theorem %

\begin{theorem}[Fite--Leighton--Wintner Theorem]\label{leiwinsca}
	\index{Fite--Leighton--Wintner theorem}
Let $L$ be as in \eqref{saeq}. Assume $\mathcal{I}=[a,\infty)$, $p>0$ on $\mathcal{I}$, and 
\begin{equation}\label{LW}
  \int_{a}^{\infty} \frac{e_0(t,a)}{p(t)} d_{\alpha}t = \int_{a}^{\infty}q(t) d_{\alpha}t = \infty. 
\end{equation}
Then $Lx=0$ is oscillatory on $[a,\infty)$.
\end{theorem}

\begin{proof} 
If $Lx=0$ is nonoscillatory on $[a,\infty)$, then $Lx=0$ is disconjugate in a neighborhood of $\infty$. By Theorem \ref{rdsols}, there is a dominant solution $x$ at $\infty$ satisfying
\begin{equation}\label{finite}
   \int_{t_0}^{\infty} \frac{e_0^2(t,a)}{p(t)x^2(t)}d_{\alpha}t < \infty
\end{equation}
for $t_0\in\mathcal{I}$ sufficiently large. Define $z$ by the Riccati substitution \eqref{zform} on $[t_0,\infty)$. Then, by Theorem \ref{saricc}, $z$ is a solution of the Riccati equation \eqref{ricceq} on $[t_0,\infty)$. It follows for all $s\ge t_0$ that
\[ \frac{D^{\alpha}[e_0(\cdot,t_0)z](s)}{e_0(s,t_0)} \stackrel{\eqref{dfminuskf}}{=} D^{\alpha}z(s) - \kappa_1(s)z(s) \stackrel{\eqref{ricceq}}{=} -q(s) - \frac{z^2(s)}{p(s)} \le -q(s), \]
whence
\[ D^{\alpha}[e_0(\cdot,t_0)z](s)e_0(t,s) \le -e_0(t,t_0)q(s).  \]
Integrating both sides of this last inequality from $t_0$ to $t$ and using Theorem \ref{ftc}, we get 
\[ z(t) - z(t_0) \le -\int_{t_0}^t q(s) d_{\alpha}s. \]
It follows that $\lim_{t\to\infty} z(t)=-\infty$. But then there exists $t_1\in \mathcal{I}$ sufficiently large so that $t_1\ge t_0$ and
\[ z(t) = \frac{p(t)D^{\alpha}x(t)}{x(t)} < 0 \quad\text{for all}\quad t\ge t_1. \]
Without loss of generality we can assume that $x(t)>0$ for $t\in[t_1,\infty)$; it then follows from Theorem \ref{incdectest} that $x$ is a positive $\alpha$-decreasing function on $[t_1,\infty)$, so that
\[ x^2(t) < e_0(t,t_1)x^2(t_1), \quad t\ge t_1. \]
But then we have that
\[ \int_{a}^{\infty}\frac{e_0^2(t,a)}{p(t)x^2(t)}d_{\alpha}t \ge \int_{t_1}^{\infty}\frac{e_0^2(t,a)}{p(t)x^2(t)}d_{\alpha}t > \frac{e_0(t_1,a)}{x^2(t_1)} \int_{t_1}^{\infty} \frac{e_0(t,a)}{p(t)}d_{\alpha}t = \infty, \]
which contradicts \eqref{finite}. This completes the proof.
\end{proof}

%%%%%%%%%%%%%%%%%%%%%%%%%%%%%%%%%%%%%%%%%%%%%%%%%%%%%%%%%%%%%%%%%%%%%%%%%
\section{Boundary Value Problems and Green's Function}\label{s33}
%%%%%%%%%%%%%%%%%%%%%%%%%%%%%%%%%%%%%%%%%%%%%%%%%%%%%%%%%%%%%%%%%%%%%%%%%

In this section we are concerned with Green's functions for a general two-point boundary value problem (abbreviated by BVP) for $Lx=0$, $L$ given in \eqref{saeq}.

%%%%%%%%%%%%%
% THEOREM   %
%%%%%%%%%%%%%
\begin{theorem}[Existence and Uniqueness of Solutions for General Two Point BVPs]\label{uniq}
	\index{existence uniqueness theorem!BVP}
	\index{boundary value problem!general two point}
Assume that the homogeneous boundary value problem
\begin{equation}\label{he}
   Lx=0,\quad \xi x(a)-\beta D^{\alpha}x(a)
   	 = \gamma x(b)+\delta D^{\alpha}x(b)=0
\end{equation}
has only the trivial solution. Then the nonhomogeneous BVP
\begin{equation}\label{nhe}
  Lx=h(t),\quad \xi x(a)-\beta D^{\alpha}x(a)=A,\quad
   	\gamma x(b)+\delta D^{\alpha}x(b)=B,
\end{equation}
where $A$ and $B$ are given constants and $h$ is continuous, has a unique solution.
\end{theorem}

\begin{proof}
The proof is similar to the classical ($\alpha=1$) case and thus is omitted.
\end{proof}

In the next example we give a BVP of the type \eqref{he} which does not have just the trivial solution. In Example \ref{zrsjsery} we give a necessary and sufficient condition for some boundary value problems of the form \eqref{he} to have only the trivial solution.

\begin{example}\label{arhuyer} 
Find all solutions of the BVP 
\begin{equation}\label{bvpexexex}
 D^{\alpha}[pD^{\alpha}x](t) = 0, \quad D^{\alpha}x(a)=D^{\alpha}x(b)=0,
\end{equation}
where $a<b$. The BVP \eqref{bvpexexex} is equivalent to a BVP of the form \eqref{he} if we take $q(t)\equiv 0$, $\xi=\gamma=0$, and $\beta=\delta=1$.
A general solution of the boundary value problem in \eqref{bvpexexex} is
\begin{equation}\label{homoqsol}
 x(t) = c_1e_0(t,a) + c_2 e_0(t,a)\int_a^t \frac{1}{p(\tau)} d_{\alpha}\tau, \quad t\in[a,b],
\end{equation}
and the boundary conditions lead to the equations
\[ D^{\alpha}x(a) = \frac{c_2}{p(a)} = 0 \quad\text{and}\quad
   D^{\alpha}x(b) = \frac{c_2e_0(b,a)}{p(b)} = 0. \]
Thus $c_2=0$, and $x(t)=c_1e_0(t,a)$ solves \eqref{bvpexexex} for any constant $c_1\in\R$. $\hfill\triangle$
\end{example}

% Example %

\begin{example}\label{zrsjsery} 
Let
\[ D = \xi\gamma\int_a^{b} \frac{1}{p(\tau)} d_{\alpha}\tau + \frac{\beta\gamma}{p(a)} + \frac{\xi\delta}{p(b)}. \] 
Using \eqref{homoqsol}, it is straightforward to show that the BVP \eqref{he} with $q\equiv  0$ has only the trivial solution if and only if $D\neq 0$.
In particular, $D=0$ for the BVP in Example \ref{arhuyer} and thus \eqref{bvpexexex} has nontrivial solutions, as seen above.  $\hfill\triangle$
\end{example}

%%%%%%%%%%%%%
% THEOREM   %
%%%%%%%%%%%%%

\begin{theorem}[Green's Function for General Two Point BVPs]\label{oldtheorem5}
	\index{Green's function!general two point BVP}
	\index{boundary value problem!general two point}
Assume that the BVP \eqref{he} has only the trivial solution. For each fixed $s\in[a,b]$, let $u(\cdot,s)$ be the unique solution of the BVP 
\begin{gather}
  Lu(\cdot,s) = 0,\quad \xi u(a,s) -\beta D^{\alpha}u(a,s) = 0,\nonumber \\
	  \gamma u(b,s) + \delta D^{\alpha}u(b,s) = -\gamma x(b,s)-\delta D^{\alpha}x(b,s),\label{unhbc}
\end{gather}
where $x(t,s)$ is the Cauchy function \eqref{eq31} for $Lx=0$, $L$ in \eqref{saeq}. Define Green's function $G:[a,b]\times[a,b]\to\R$ 
for the BVP \eqref{he} by
\[ G(t,s)=\begin{cases} u(t,s) &: a\le t \le s \le b, \\
                        v(t,s) &: a\le s \le t \le b, \end{cases} \]
where $v(t,s):=u(t,s)+x(t,s)$ for $t,s\in[a,b]$. Then for each fixed $s\in[a,b]$, $v(\cdot,s)$ is a solution of $Lx=0$ and satisfies the second boundary condition in \eqref{he}. If $h$ is continuous, then 
\[ x(t):=\int_a^{b}G(t,s)h(s) d_{\alpha}s \]
is the solution of the nonhomogeneous BVP \eqref{nhe} with $A=B=0$.
\end{theorem}

\begin{proof} 
The existence and uniqueness of $u(t,s)$ is guaranteed by Theorem \ref{uniq}. Since for each fixed $s\in [a,b]$, $u(\cdot,s)$ and $x(\cdot,s)$ are 
solutions of $Lx=0$, $L$ in \eqref{saeq}, we have for each fixed $s\in[a,b]$ that $v(\cdot,s)=u(\cdot,s)+x(\cdot,s)$ is also a solution of $Lx=0$.
It follows from \eqref{unhbc} that $v(\cdot,s)$ satisfies the second boundary condition in \eqref{he} for each fixed $s\in[a,b]$.
First note that
\begin{eqnarray*}
  x(t) &=& \int_a^{b}G(t,s)h(s) d_{\alpha}s\\
  &=& \int_a^tG(t,s)h(s) d_{\alpha}s+\int_t^{b}G(t,s)h(s) d_{\alpha}s\\
  &=& \int_a^tv(t,s)h(s) d_{\alpha}s
      +\int_t^{b}u(t,s)h(s) d_{\alpha}s\\ 
  &=& \int_a^t[u(t,s)+x(t,s)]h(s) d_{\alpha}s
      +\int_t^{b}u(t,s)h(s) d_{\alpha}s\\
  &=& \int_a^{b}u(t,s)h(s) d_{\alpha}s
      +\int_a^tx(t,s)h(s) d_{\alpha}s \\
  &=& \int_a^{b}u(t,s)h(s) d_{\alpha}s+w(t),
\end{eqnarray*}
where, by the variation of constants formula in Theorem \ref{t31}, $w$ is the solution of the IVP
\begin{equation}\label{wcondh}
 Lw = h(t),\quad w(a)=D^{\alpha}w(a)=0.
\end{equation}
It follows that
\[ Lx(t) = \int_a^{b}Lu(t,s)h(s) d_{\alpha}s + Lw(t) = Lw(t) = h(t). \]
Hence $x$ is a solution of the nonhomogeneous equation $Lx=h(t)$. It remains to show that $x$ satisfies the two boundary conditions 
in \eqref{he}. Now
\begin{eqnarray*}
   \xi x(a)-\beta D^{\alpha}x(a) 
   &=& \int_a^{b}[\xi u(a,s)-\beta D^{\alpha}u(a,s)]h(s) d_{\alpha}s = 0,
\end{eqnarray*}
since for each fixed $s$, $u(\cdot,s)$ satisfies the first boundary condition in \eqref{he} and $w$ satisfies \eqref{wcondh}. Hence $x$ satisfies the first boundary condition in \eqref{he}. From earlier in this proof,
\begin{eqnarray*} 
x(t)&=&\int_a^{b}u(t,s)h(s) d_{\alpha}s+\int_a^tx(t,s)h(s) d_{\alpha}s\\
  &=&\int_a^{b}v(t,s)h(s) d_{\alpha}s-\int_t^{b}x(t,s)h(s) d_{\alpha}s\\
  &=&\int_a^{b}v(t,s)h(s) d_{\alpha}s+\int_{b}^tx(t,s)h(s) d_{\alpha}s\\
  &=&\int_a^{b}v(t,s)h(s) d_{\alpha}s+y(t),
\end{eqnarray*}
where, by the variation of constants formula in Theorem \ref{t31}, $y$ solves
\begin{equation}\label{ycondh} 
 Ly = h(t),\quad y(b)=D^{\alpha}y(b)=0. 
\end{equation}
Then
\begin{eqnarray*}
 \gamma x(b)+\delta D^{\alpha}x(b)
   &=& \int_a^{b}\left[\gamma v(b,s) + \delta D^{\alpha}v(b,s)\right]h(s) d_{\alpha}s = 0,
\end{eqnarray*}
since for each fixed $s$, $v(\cdot,s)$ satisfies the second boundary condition in \eqref{he} and $y$ satisfies \eqref{ycondh}. Hence $x$ satisfies the second boundary condition in \eqref{he}. The uniqueness of solutions of the nonhomogeneous BVP \eqref{nhe} (with $A=B=0$) follows from Theorem \ref{uniq}.
\end{proof}

Instead of the Cauchy function approach as above, in the next theorem we find another form of Green's function for the BVP \eqref{he}; of particular interest is a symmetry condition on the square $[a,b]\times[a,b]$ satisfied by Green's function.

% Theorem %

\begin{theorem}\label{2ptbvpthm}
	\index{Green's function!general two point BVP}
Assume that the BVP \eqref{he} has only the trivial solution. Let $\phi$ be the solution of the IVP
\[ L\phi = 0, \quad \phi(a) = \beta, \quad D^{\alpha}\phi(a)=\xi, \]
and let $\psi$ be the solution of 
\[ L\psi = 0, \quad \psi(b) = \delta, \quad D^{\alpha}\psi(b)=-\gamma. \]
Then Green's function for the BVP \eqref{he} is given by
\begin{eqnarray}\label{formgf2} G(t,s) = \frac{1}{p(s)W(\phi,\psi)(s)}\begin{cases}
   \phi(t)\psi(s)&: a\le t\le s \le b, \\
	 \psi(t)\phi(s)&: a\le s\le t \le b, \end{cases}
\end{eqnarray}
where $W$ is the Wronskian \eqref{wdef}. Furthermore, $G$ satisfies the property
\[ e_0(s,t)G(t,s) = e_0(t,s)G(s,t). \]
\end{theorem}

\begin{proof}
Let $\phi$ and $\psi$ be as stated in the theorem. We use 
Theorem \ref{oldtheorem5} to prove that $G$ defined by \eqref{formgf2} is Green's function for the BVP \eqref{he}. Note that
\[ \xi\phi(a)-\beta D^{\alpha}\phi(a)=\xi\beta-\beta\xi=0 \]
and
\[ \gamma\psi(b)+\delta D^{\alpha}\psi(b) = \gamma\delta-\delta\gamma=0. \]
Hence $\phi$ and $\psi$ satisfy the first and second boundary condition in \eqref{he}, respectively. Let
\[ u(t,s):=\frac{\phi(t)\psi(s)}{p(s)W(\phi,\psi)(s)} \quad\text{and}\quad
   v(t,s):=\frac{\psi(t)\phi(s)}{p(s)W(\phi,\psi)(s)} \]
for $t\in[a,b]$, $s\in[a,b]$. Note that for each fixed $s\in[a,b]$, $u(\cdot,s)$ and $v(\cdot,s)$ are solutions of $Lx=0$, $L$ in \eqref{saeq}, on $[a,b]$. Also for each fixed $s\in[a,b]$,
\[ \xi u(a,s)-\beta D^{\alpha}u(a,s) = \frac{\psi(s)}{p(s)W(\phi,\psi)(s)}[\xi\phi(a)-\beta D^{\alpha}\phi(a)]=0 \]
and
\[ \gamma v(b,s)+\delta D^{\alpha}v(b,s) = \frac{\phi(s)}{p(s)W(\phi,\psi)(s)}[\gamma\psi(b)+\delta D^{\alpha}\psi(b)]=0. \]
Hence for each fixed $s\in [a,b]$, $u(\cdot,s)$ and $v(\cdot,s)$ satisfy the first and second boundary condition in \eqref{he}, respectively. Let
\[ \chi(t,s):= v(t,s)-u(t,s) = \frac{\psi(t)\phi(s)-\phi(t)\psi(s)}{p(s)W(\phi,\psi)(s)}. \] 
It follows that for each fixed $s$, $\chi(\cdot,s)$ is a solution of $Lx=0$, $\chi(s,s)=0$, and
\begin{eqnarray*}
 D^{\alpha}\chi(s,s)
  &=& \frac{\phi(s)D^{\alpha}\psi(s) - \psi(s)D^{\alpha}\phi(s)}{p(s)W(\phi,\psi)(s)} = \frac{1}{p(s)}.
\end{eqnarray*}
Consequently $\chi(t,s)=x(t,s)$ is the Cauchy function for $Lx=0$, and we have
	\index{Cauchy function!second order equation}
\[ v(t,s) = u(t,s)+x(t,s). \]
It remains to prove that for each fixed $s$, $u(\cdot,s)$ satisfies \eqref{unhbc}. To see this consider
\begin{eqnarray*}
\gamma u(b,s)+\delta D^{\alpha}u(b,s) &=& \gamma v(b,s)+\delta D^{\alpha}v(b,s) - [\gamma x(b,s)+\delta D^{\alpha}x(b,s)] \\
  &=& -\gamma x(b,s)-\delta D^{\alpha}x(b,s).
\end{eqnarray*}
Hence by Theorem \ref{oldtheorem5}, $G(t,s)$ defined by \eqref{formgf2} is Green's function for \eqref{he}. It follows from Abel's formula (Corollary \ref{wronskid}) and \eqref{formgf2} that $G$ satisfies the condition $e_0(s,t)G(t,s)=e_0(t,s)G(s,t)$ for $s,t\in[a,b]$.
\end{proof}

%%%%%%%%%%%%%%%%%%%%%%%%%%%%%%%%
% Subsection Conjugate Problem %
%%%%%%%%%%%%%%%%%%%%%%%%%%%%%%%%
\subsection{Conjugate Problem and Disconjugacy}

In this subsection we examine Theorem \ref{oldtheorem5} and Theorem \ref{2ptbvpthm} in more detail, in particular for the special case where the boundary conditions \eqref{unhbc} are conjugate (also known as Dirichlet) boundary conditions.

%%%%%%%%%%%%%%%%%%%
%  Corollary 11.6 %
%%%%%%%%%%%%%%%%%%%

\begin{corollary}[Green's Function for the Conjugate Problem]\label{oldcor6}
 \index{Green's function!conjugate problem}
Assume the BVP
\begin{equation}\label{cojubvp}
  Lx=0,\quad x(a)=x(b)=0.
\end{equation}
has only the trivial solution. Let $x(t,s)$ be the Cauchy function for $Lx=0$, $L$ in \eqref{saeq}. For each fixed $s\in\mathcal{I}$, let $u(\cdot,s)$ be the unique solution of the BVP
\[ Lu(\cdot,s)=0,\quad u(a,s)=0,\quad u(b,s)=-x(b,s). \]
Then
\[ G(t,s)=\begin{cases}
    u(t,s) &: t \le s, \\
    v(t,s) &: t \ge s, \end{cases} \]
where $v(t,s)=u(t,s)+x(t,s)$, is Green's function for the BVP \eqref{cojubvp}.
Moreover, for each fixed $s\in [a,b]$, $v(\cdot,s)$ is a solution of $Lx=0$ and $v(b,s)=0$.
\end{corollary}
\begin{proof}
This corollary follows from Theorem \ref{oldtheorem5} with
$\xi=\gamma=1$ and $\beta=\delta=0$.
\end{proof}

%%%%%%%%%%%%%%%%%%
% Corollary 11.7 %
%%%%%%%%%%%%%%%%%%

\begin{corollary}\label{oldcor7}
Green's function for the BVP \eqref{cojubvp} with $q\equiv 0$ is given by
\[ G(t,s) = \frac{-e_0(t,s)}{\int_a^{b} \frac{1}{p(\tau)}d_{\alpha}\tau}
\begin{cases}
  \displaystyle\int_a^t \frac{1}{p(\tau)} d_{\alpha}\tau \int_s^b \frac{1}{p(\tau)} d_{\alpha}\tau &: a\le t\le s\le b, \\
  \displaystyle\int_a^s \frac{1}{p(\tau)} d_{\alpha}\tau \int_t^b \frac{1}{p(\tau)} d_{\alpha}\tau &: a\le s\le t\le b. \end{cases} \]
\end{corollary}

\begin{proof}
It is easy to check that the BVP
\[ D^{\alpha}[pD^{\alpha}x](t) = 0, \qquad x(a)=x(b)=0, \]
has only the trivial solution from a modification of Example \ref{arhuyer}. By Example \ref{oldex1}, the Cauchy function for 
\begin{equation}\label{qiszero}
 D^{\alpha}[pD^{\alpha}x](t) = 0 
\end{equation}
is given by
\[ x(t,s) = e_0(t,s)\int_{s}^t \frac{1}{p(\tau)} d_{\alpha}\tau. \]
By Corollary \ref{oldcor6}, $u(\cdot,s)$ from the statement of Corollary \ref{oldcor6} solves \eqref{qiszero} for each fixed $s\in[a,b]$ and satisfies
\begin{equation}\label{baundaricondischons}
   u(a,s) = 0 \quad\text{and}\quad
   u(b,s) = -x(b,s) = - e_0(b,s)\int_{s}^b \frac{1}{p(\tau)} d_{\alpha}\tau.
\end{equation}
Since
\[ x_1(t) = e_0(t,a) \quad\text{and}\quad x_2(t) = e_0(t,a)\int_a^t \frac{1}{p(\tau)} d_{\alpha}\tau \]
are solutions of \eqref{qiszero},
\[ u(t,s) = \xi(s)e_0(t,a)+\beta(s)e_0(t,a)\int_a^t \frac{1}{p(\tau)} d_{\alpha}\tau. \]
Using the boundary conditions \eqref{baundaricondischons}, it can be shown that
\[ u(t,s) = -e_0(t,s)\frac{\int_a^t\frac{1}{p(\tau)} d_{\alpha}\tau \int_{s}^{b}\frac{1}{p(\tau)} d_{\alpha}\tau}{\int_a^{b}\frac{1}{p(\tau)} d_{\alpha}\tau}. \]
Hence $G(t,s)$ has the desired form for $t\leq s$.
By Corollary \ref{oldcor6} for $t\ge s$,
\[ G(t,s)=x(t,s)+u(t,s). \]
Therefore for $t\ge s$,
\begin{eqnarray*}
 G(t,s) &=& e_0(t,s)\int_s^t \frac{1}{p(\tau)} d_{\alpha}\tau -e_0(t,s)\frac{\int_a^t\frac{1}{p(\tau)} d_{\alpha}\tau \int_{s}^{b}\frac{1}{p(\tau)} d_{\alpha}\tau}{\int_a^{b}\frac{1}{p(\tau)} d_{\alpha}\tau} \\
   &=& e_0(t,s)\frac{\int_{s}^t \frac{1}{p(\tau)} d_{\alpha}\tau \int_a^{b}\frac{1}{p(\tau)} d_{\alpha}\tau
     -\int_a^t\frac{1}{p(\tau)} d_{\alpha}\tau \int_{s}^{b}\frac{1}{p(\tau)} d_{\alpha}\tau}{\int_a^{b}\frac{1}{p(\tau)} d_{\alpha}\tau} \\
   &=& e_0(t,s)\frac{\int_a^t\frac{1}{p(\tau)} d_{\alpha}\tau
       \left[-\int_t^{b}\frac{1}{p(\tau)} d_{\alpha}\tau\right]+\int_{s}^t\frac{1}{p(\tau)} d_{\alpha}\tau
       \int_t^{b}\frac{1}{p(\tau)} d_{\alpha}\tau}{\int_a^{b}\frac{1}{p(\tau)} d_{\alpha}\tau} \\
   &=& -e_0(t,s)\frac{\int_a^{s}\frac{1}{p(\tau)} d_{\alpha}\tau \int_t^{b}\frac{1}{p(\tau)} d_{\alpha}\tau}{\int_a^{b}\frac{1}{p(\tau)} d_{\alpha}\tau},
\end{eqnarray*}
which is the desired result.
\end{proof}

The following corollary follows immediately from Corollary \ref{oldcor7} by taking $p(\tau)\equiv 1$.

% Corollary 11.8 %

\begin{corollary}\label{grcp} 
	\index{Green's function!conjugate problem}
Green's function for the BVP 
\[ D^{\alpha}D^{\alpha}x=0, \quad x(a)=x(b)=0 \]
is given by
\[ G(t,s) = \frac{-e_0(t,s)}{\int_a^{b} 1 d_{\alpha}\tau}
\begin{cases}
  \displaystyle\int_a^t 1 d_{\alpha}\tau \int_{s}^{b} 1 d_{\alpha}\tau &: a\le t\le s\le b, \\
  \displaystyle\int_a^s 1 d_{\alpha}\tau\int_t^{b} 1 d_{\alpha}\tau &: a\le s\le t\le b. \end{cases} \]
For example, if $\kappa_1(t)=(1-\alpha)(\omega t)^{\alpha}$ and $\kappa_0(t)=\alpha (\omega t)^{1-\alpha}$ for $\alpha\in(0,1]$ and $\omega,t\in(0,\infty)$, then
\[ \int_a^t 1 d_{\alpha}\tau = \frac{t^{\alpha}-a^{\alpha}}{\alpha^2\omega^{1-\alpha}} \quad\text{and}\quad e_0(t,s)=e^{-\left(\frac{1-\alpha}{2\alpha^2}\right)\omega^{2\alpha-1}\left(t^{2\alpha}-s^{2\alpha}\right)} \]
for $a,s\in(0,\infty)$, and
\[ G(t,s) = \frac{-e^{-\left(\frac{1-\alpha}{2\alpha^2}\right)\omega^{2\alpha-1}\left(t^{2\alpha}-s^{2\alpha}\right)}}{\alpha^2\omega^{1-\alpha}\left(b^{\alpha}-a^{\alpha}\right)} \begin{cases} \left(t^{\alpha}-a^{\alpha}\right)\left(b^{\alpha}-s^{\alpha}\right) &: a\le t\le s\le b, \\ \left(s^{\alpha}-a^{\alpha}\right)\left(b^{\alpha}-t^{\alpha}\right) &: a\le s\le t\le b. \end{cases} \]
\end{corollary}

The following corollary follows immediately from Corollary \ref{oldcor7} by taking 
\begin{equation}\label{pinvconj}
 p^{-1}(\tau) = e_0(b,\tau)D^{\alpha}\left[\frac{\tau}{e_0(b,\tau)}\right].
\end{equation}

% Corollary 11.9 %

\begin{corollary}
Let $a,b\in\R$ with $a<b$, and let $p$ be given via \eqref{pinvconj}. Then Green's function for the BVP 
\[ D^{\alpha}\left[pD^{\alpha}x\right](t)=0, \quad x(a)=x(b)=0 \]
is given by
\[ G(t,s) = \frac{-e_0(t,s)}{b-a}
\begin{cases}
  (t-a)(b-s) &: a\le t\le s\le b, \\
  (s-a)(b-t) &: a\le s\le t\le b. \end{cases} \]
\end{corollary}

\begin{proof}
By Corollary \ref{oldcor7}, Green's function for the BVP \eqref{cojubvp} with $q\equiv 0$ is given by
\[ G(t,s) = \frac{-e_0(t,s)}{\int_a^{b} \frac{1}{p(\tau)}d_{\alpha}\tau}
\begin{cases}
  \displaystyle\int_a^t \frac{1}{p(\tau)} d_{\alpha}\tau \int_s^b \frac{1}{p(\tau)} d_{\alpha}\tau &: a\le t\le s\le b, \\
  \displaystyle\int_a^s \frac{1}{p(\tau)} d_{\alpha}\tau \int_t^b \frac{1}{p(\tau)} d_{\alpha}\tau &: a\le s\le t\le b. \end{cases} \]
Using \eqref{pinvconj} and Theorem \ref{ftc}, we see that
\[ \int_a^{b} \frac{1}{p(\tau)}d_{\alpha}\tau = \int_a^{b} D^{\alpha}\left[\frac{\tau}{e_0(b,\tau)}\right]e_0(b,\tau)d_{\alpha}\tau = b-a. \]
Similar evaluations of the other integrals yield the result.
\end{proof}

The proofs of the following two theorems are similar to their classical ($\alpha=1$) counterparts and thus are omitted.

%%%%%%%%%%%%%%%%
% Theorem 11.9 %
%%%%%%%%%%%%%%%%

\begin{theorem}\label{negG} 
Let $L$ be as in \eqref{saeq}. If $p>0$ and $Lx=0$ is disconjugate on $[a,b]$, then Green's function for the conjugate BVP \eqref{cojubvp} exists and satisfies
\[ G(t,s)<0 \quad\text{for}\quad t,s\in(a,b). \]
\end{theorem}

%%%%%%%%%%%%%%%%%%%
%  Theorem 11.10  %
%%%%%%%%%%%%%%%%%%%

\begin{theorem}[Comparison Theorem for Conjugate BVPs]\label{oldtheorem8}
	\index{comparison theorem!for BVPs}
Let $L$ be as in \eqref{saeq}. Let $p>0$ and assume that $Lx=0$ is disconjugate on $[a,b]$. 
If $u,v\in\D$ satisfy
\[ Lu(t)\leq Lv(t)\;\mbox{ for all }\;t\in[a,b],\quad
   u(a)\geq v(a),\quad u(b)\geq v(b), \]
then
\[ u(t)\geq v(t) \quad\text{for all}\quad t\in[a,b]. \]
\end{theorem}

%%%%%%%%%%%%%%%%%%%%%%%%%%%%%%%%
%  Subsection  Focal  Problem  %
%%%%%%%%%%%%%%%%%%%%%%%%%%%%%%%%
\subsection{Right Focal Problem}

Similar to the subsection above on the conjugate boundary conditions and disconjugacy, here we examine Theorem \ref{oldtheorem5} and Theorem \ref{2ptbvpthm} for the special case where the boundary conditions \eqref{unhbc} are right focal boundary conditions, namely the boundary value problem
\begin{equation}\label{fokalbvp}
   Lx=0, \quad x(a) = D^{\alpha}x(b)=0. 
\end{equation}

%%%%%%%%%%%%%%%%%%%
% COROLLARY 11.11 %
%%%%%%%%%%%%%%%%%%%

\begin{corollary}[Green's Function for Focal BVPs]\label{oldcor11}
	\index{Green's function!focal BVP}
Assume that the BVP \eqref{fokalbvp} has only the trivial solution. For each fixed $s\in[a,b]$, let $u(\cdot,s)$ be the solution of the BVP 
\[ Lu(\cdot,s)=0,\quad u(a,s)=0,\quad 
   D^{\alpha}u(b,s) = -D^{\alpha}x(b,s), \]
where $x(t,s)$ is the Cauchy function of $Lx=0$. Then
\[ G(t,s)=\begin{cases}
	   u(t,s)        &: a\le t\le s\le b \\
	   u(t,s)+x(t,s) &: a\le s\le t\le b \end{cases} \]
is Green's function for the right focal BVP \eqref{fokalbvp}.
\end{corollary}

\begin{proof}
This follows from Theorem \ref{oldtheorem5} with 
$\xi=\delta=1$ and $\beta=\gamma=0$.
\end{proof}

%%%%%%%%%%%%%%%%%%%
% Corollary 11.12 %
%%%%%%%%%%%%%%%%%%%

\begin{corollary}\label{ggff}
Green's function for the focal BVP \eqref{fokalbvp} with $q\equiv 0$ is given by
\[ G(t,s) = -e_0(t,s)\begin{cases}
  \int_a^t\frac{1}{p(\tau)} d_{\alpha}\tau   &: a\le t\le s\le b, \\
	\int_a^{s}\frac{1}{p(\tau)} d_{\alpha}\tau &: a\le s\le t\le b. \end{cases} \]
\end{corollary}

\begin{proof}
It is easy to see that \eqref{fokalbvp} with $q\equiv 0$ has only the trivial solution. Hence we can apply Corollary \ref{oldcor11} to find the focal 
Green's function $G(t,s)$. For $t\le s$, $G(t,s)=u(t,s)$, where for each fixed $s$, $u(\cdot,s)$ solves the BVP
\[ Lu(\cdot,s)=0,\quad
   u(a,s)=0, \quad D^{\alpha}u(b,s) = -D^{\alpha}x(b,s), \]
and where $x(t,s)$ is the Cauchy function for \eqref{qiszero}. Solving this BVP, we get that
\[ u(t,s) = -e_0(t,s)\int_a^t\frac{1}{p(\tau)} d_{\alpha}\tau, \]
which is the desired expression for $G(t,s)$ if $t\le s$. 
If $t\ge s$, then
\begin{eqnarray*}
 G(t,s) &=& u(t,s)+x(t,s) = -e_0(t,s)\int_a^{s}\frac{1}{p(\tau)} d_{\alpha}\tau.
\end{eqnarray*}
This completes the proof.
\end{proof}

%%%%%%%%%%%%%%%%%%%
% Corollary 11.13 % 
%%%%%%%%%%%%%%%%%%%

\begin{corollary}\label{gf}
Green's function for the focal BVP \eqref{fokalbvp} with $p(t)\equiv 1$ and $q\equiv 0$ is given by
\[ G(t,s) = -e_0(t,s)\begin{cases}
  \int_a^t 1 d_{\alpha}\tau &: a\le t\le s\le b, \\
	\int_a^s 1 d_{\alpha}\tau &: a\le s\le t\le b. \end{cases} \]
For example, if $\kappa_1(t)=(1-\alpha)(\omega t)^{\alpha}$ and $\kappa_0(t)=\alpha (\omega t)^{1-\alpha}$ for $\alpha\in(0,1]$ and $\omega,t\in(0,\infty)$, then
\[ \int_a^t 1 d_{\alpha}\tau = \frac{t^{\alpha}-a^{\alpha}}{\alpha^2\omega^{1-\alpha}} \quad\text{and}\quad e_0(t,s)=e^{-\left(\frac{1-\alpha}{2\alpha^2}\right)\omega^{2\alpha-1}\left(t^{2\alpha}-s^{2\alpha}\right)} \]
for $a,s\in(0,\infty)$, and
\[ G(t,s) = \frac{-e^{-\left(\frac{1-\alpha}{2\alpha^2}\right)\omega^{2\alpha-1}\left(t^{2\alpha}-s^{2\alpha}\right)}}{\alpha^2\omega^{1-\alpha}} 
\begin{cases} t^{\alpha}-a^{\alpha} &: t \le s, \\ s^{\alpha}-a^{\alpha} &: t\ge s. \end{cases} \]
\end{corollary}

\begin{proof}
Put $p(\tau)\equiv 1$ in Corollary \ref{ggff}.
\end{proof}

% Corollary 11.15 %

\begin{corollary}
Let $a,b\in\R$ with $a<b$, and let $p$ be given via \eqref{pinvconj}. Then Green's function for the BVP 
\[ D^{\alpha}\left[pD^{\alpha}x\right](t)=0, \quad x(a)=D^{\alpha}x(b)=0 \]
is given by
\[ G(t,s) = -e_0(t,s)
\begin{cases}
  t-a &: a\le t\le s\le b, \\
  s-a &: a\le s\le t\le b. \end{cases} \]
\end{corollary}

%%%%%%%%%%%%%%%%%%%%%%%%%%%%%%%%
%  Subsection Periodic Problem %
%%%%%%%%%%%%%%%%%%%%%%%%%%%%%%%%
\subsection{Periodic Problem}

Finally we consider periodic boundary conditions. In traditional ($\alpha=1$) calculus, the geometry of periodicity means returning to the same values in the sense that they lie along the same horizontal straight line (geodesic with slope zero). Considering the context of the derivative \eqref{derivdef}, we investigate the periodic BVP 
	\index{boundary value problem!periodic}
	\index{periodic boundary conditions}
\begin{equation}\label{hperiodicbvp}
    Lx=0,\quad x(a) = e_0(a,b)x(b),\quad D^{\alpha}x(a) = e_0(a,b)D^{\alpha}x(b).
\end{equation}

% Theorem %

\begin{theorem}\label{perasefyu} 
	\index{existence uniqueness theorem!BVP}
Let $L$ be as in \eqref{saeq}. Assume that the homogeneous periodic BVP \eqref{hperiodicbvp} has only the trivial solution. Then for $t\in[a,b]$, the nonhomogeneous BVP
\begin{equation}\label{nhperiodicbvp}
  Lx=h(t),\quad x(a)-e_0(a,b)x(b)=A, \quad D^{\alpha}x(a)-e_0(a,b)D^{\alpha}x(b)=B,
\end{equation}
where $A$ and $B$ are given constants and $h$ is continuous, has a unique solution.
\end{theorem}

\begin{proof} 
Let $x_1$ and $x_2$ be linearly independent solutions of $Lx=0$. Then
\[ x(t) = c_1 x_1(t) + c_2 x_2(t) \]
is a general solution of $Lx=0$. Note that $x$ satisfies the boundary conditions in \eqref{hperiodicbvp} if and only if $c_1$ and $c_2$ are constants satisfying
\[ M\vecc{c_1}{c_2}=0 \]
with
\[ M = \mat{x_1(a)-e_0(a,b)x_1(b)}{x_2(a)-e_0(a,b)x_2(b)}{D^{\alpha}x_1(a)-e_0(a,b)D^{\alpha}x_1(b)}{D^{\alpha}x_2(a)-e_0(a,b)D^{\alpha}x_2(b)}. \]
Since we are assuming that \eqref{hperiodicbvp} has only the trivial solution, it follows that 
\[ c_1=c_2=0 \]
is the unique solution of the above linear system. Hence
\begin{equation}\label{detjseooper}
  \det M\ne 0.
\end{equation}
Now we show that \eqref{nhperiodicbvp} has a unique solution. Let $u_0$ be a fixed solution of $Lu=h(t)$. Then a general solution of $Lu=h(t)$ is given by
\[ u(t) = a_1x_1(t)+a_2x_2(t)+u_0(t). \]
It follows that $u$ satisfies the boundary conditions in \eqref{nhperiodicbvp} if and only if $a_1$ and $a_2$ are constants satisfying the system of equations
\[ M\vecc{a_1}{a_2} = \vecc{A - u_0(a) + e_0(a,b)u_0(b)}{B - D^{\alpha}u_0(a) + e_0(a,b)D^{\alpha}u_0(b)}. \]
This system has a unique solution because of \eqref{detjseooper}, and hence \eqref{nhperiodicbvp} has a unique solution.
\end{proof}

%%%%%%%%%%%%%%%%%%%%%%%%%%%%
% Theorem Green's function %
%%%%%%%%%%%%%%%%%%%%%%%%%%%%

\begin{theorem}[Green's Function for Periodic BVPs]\label{tgfper}
	\index{Green's function!periodic BVP} 
Assume that the homogeneous BVP \eqref{hperiodicbvp} has only the trivial solution. For each fixed $s\in\mathcal{I}$, let $u(\cdot,s)$ be the solution of the BVP
\begin{equation}\label{bcu1per}
 \left\{\begin{array}{l}Lu(\cdot,s)=0\\
  u(a,s) = e_0(a,b)\left[u(b,s) + x(b,s)\right] \\
  D^{\alpha}u(a,s) = e_0(a,b)\left[D^{\alpha}u(b,s) + D^{\alpha}x(b,s)\right],\end{array}\right.
\end{equation}
where $x(t,s)$ is the Cauchy function for $Lx=0$, $L$ as in \eqref{saeq}. Define
\begin{equation}\label{formgfper} 
  G(t,s):=\begin{cases}
      u(t,s) &: t \le s, \\
	    v(t,s) &: t \ge s, \end{cases}
\end{equation}
where $v(t,s):=u(t,s)+x(t,s)$. If $h$ is continuous, then 
\[ x(t):=\int_a^{b}G(t,s)h(s) d_{\alpha}s \]
is the unique solution of the nonhomogeneous periodic BVP \eqref{nhperiodicbvp} with $A=B=0$. Furthermore, for each fixed $s\in [a,b]$, $v(\cdot,s)$ is a solution of $Lx=0$, and 
\[ u(a,s) = e_0(a,b)v(b,s),\quad D^{\alpha}u(a,s) = e_0(a,b)D^{\alpha}v(b,s). \]
\end{theorem}

\begin{proof} 
The existence and uniqueness of $u(t,s)$ is guaranteed by Theorem \ref{perasefyu}. Since $v(t,s)=u(t,s)+x(t,s)$, we have for each fixed $s$ that $v(\cdot,s)$ is a solution of $Lx=0$. Using the boundary conditions in \eqref{bcu1per}, it is easy to see that for each fixed $s$, $u(a,s)=e_0(a,b)v(b,s)$ and $D^{\alpha}u(a,s)=e_0(a,b)D^{\alpha}v(b,s)$. Let $G$ be as in \eqref{formgfper} and notice that
\begin{eqnarray*}
  x(t)&=&\int_a^{b}G(t,s)h(s) d_{\alpha}s = \int_a^{b}u(t,s)h(s) d_{\alpha}s + \int_a^tx(t,s)h(s) d_{\alpha}s\\
  &=&\int_a^{b}u(t,s)h(s) d_{\alpha}s + z(t),
\end{eqnarray*}
where, by the variation of constants formula in Theorem \ref{t31}, $z$ solves
\[ Lz=h(t),\quad z(a)=D^{\alpha}z(a)=0. \]
Hence
\[ Lx(t)=\int_a^{b}Lu(t,s)h(s) d_{\alpha}s+Lz(t)=Lz(t)=h(t). \]
Thus $x$ is a solution of $Lx=h(t)$. Note that
\begin{eqnarray*}
x(a) &=& \int_a^{b} u(a,s)h(s) d_{\alpha}s+z(a) = \int_a^{b} e_0(a,b)v(b,s)h(s) d_{\alpha}s\\
     &=& \int_a^{b} e_0(a,b)G(b,s)h(s) d_{\alpha}s = e_0(a,b)x(b),
\end{eqnarray*}
and
\begin{eqnarray*}
D^{\alpha}x(a) &=&\int_a^{b}D^{\alpha}u(a,s)h(s) d_{\alpha}s+D^{\alpha}z(a) = \int_a^{b}e_0(a,b)D^{\alpha}v(b,s)h(s) d_{\alpha}s\\
&=&\int_a^{b}e_0(a,b)D^{\alpha}G(b,s)h(s) d_{\alpha}s = e_0(a,b)D^{\alpha}x(b).
\end{eqnarray*}
Hence $x$ satisfies the periodic boundary conditions in \eqref{hperiodicbvp}.
\end{proof}

% Example %

\begin{example} 
Let $\alpha\in(0,1]$. Using Theorem \ref{tgfper} we will solve the periodic BVP
\begin{equation}\label{erehjr}
  D^{\alpha}D^{\alpha}x + x = 3\sin_{\alpha}(2;t,0), \quad
     x(0) = e_0(0,\pi^*)x(\pi^*), \quad D^{\alpha}x(0) = e_0(0,\pi^*)D^{\alpha}x(\pi^*)
\end{equation}
on $\mathcal{I}=[0,\pi^*]$, where $\pi^*\in\R$ is such that
\[ h_1(\pi^*,0)=\pi, \]
and where we have used the notation
\[ \sin_{\alpha}(2;t,0):=e_0(t,0)\sin\left(2h_1(t,0)\right) \]
from Example \ref{coscosh}. It is easy to show that the homogeneous BVP 
\[ D^{\alpha}D^{\alpha}x + x = 0, \quad x(0) = e_0(0,\pi^*)x(\pi^*), \quad D^{\alpha}x(0) = e_0(0,\pi^*)D^{\alpha}x(\pi^*) \]
has only the trivial solution by Example \ref{coscosh}; hence one approach to solving this would be to use Theorem \ref{tgfper} to solve \eqref{erehjr}. 
By Theorem \ref{tgfper}, Green's function $G$ is given by \eqref{formgfper}, where for each fixed $s\in[0,\pi^*]$, $u(\cdot,s)$ is the solution of $D^{\alpha}D^{\alpha}u+u=0$ with boundary conditions
\begin{eqnarray}
 u(0,s) &=& e_0(0,\pi^*)\left[u(\pi^*,s)+x(\pi^*,s)\right], \nonumber \\
 D^{\alpha}u(0,s) &=& e_0(0,\pi^*)\left[D^{\alpha}u(\pi^*,s)+D^{\alpha}x(\pi^*,s)\right], \label{sadeuhpere}
\end{eqnarray}
where $x(t,s)$ is the Cauchy function for $D^{\alpha}D^{\alpha}x+x=0$, and $v(t,s):=u(t,s)+x(t,s)$.
The Cauchy function for $D^{\alpha}D^{\alpha}x+x=0$ is given by Theorem \ref{t33} as
\[ x(t,s) = e_0(t,s)\sin(h_1(t,s)) = \sin_{\alpha}(1;t,s). \]
Since for each fixed $s$, $u(\cdot,s)$ is a solution of $D^{\alpha}D^{\alpha}u+u=0$,
\[ u(t,s) = A(s)\cos_{\alpha}(1;t,0) + B(s)\sin_{\alpha}(1;t,0). \]
From the boundary conditions in \eqref{sadeuhpere} we get
\[ A(s) = \frac{\sin_{\alpha}(1;s,0)}{2e^2_0(s,0)} \quad \text{and} \quad
   B(s) = \frac{-\cos_{\alpha}(1;s,0)}{2e^2_0(s,0)}. \]
It follows that
\[ u(t,s) = \frac{\sin_{\alpha}(1;s,0)\cos_{\alpha}(1;t,0)}{2e^2_0(s,0)} - \frac{\cos_{\alpha}(1;s,0)\sin_{\alpha}(1;t,0)}{2e^2_0(s,0)} = \frac{-1}{2}\sin_{\alpha}(1;t,s). \]
Therefore
\[ v(t,s) = u(t,s) + x(t,s) = \frac{1}{2}\sin_{\alpha}(1;t,s). \]
Hence by Theorem \ref{tgfper} the solution of the BVP \eqref{erehjr} is given by
\begin{eqnarray*}
 x(t) &=& 3\int_0^{\pi^*}G(t,s)\sin_{\alpha}(2;s,0) d_{\alpha}s \\
 &=& \frac{3}{2}\int_0^{t}\sin_{\alpha}(1;t,s)\sin_{\alpha}(2;s,0)d_{\alpha}s - \frac{3}{2}\int_t^{\pi^*} \sin_{\alpha}(1;t,s)\sin_{\alpha}(2;s,0)d_{\alpha}s \\
 &=& \frac{3}{2}\sin(h_1(t,0)) \int_0^{t} \cos_{\alpha}(1;s,0)\sin(2h_1(s,0))e_0(t,s)d_{\alpha}s \\
 & & -\frac{3}{2}\cos(h_1(t,0)) \int_0^{t} \sin_{\alpha}(1;s,0)\sin(2h_1(s,0))e_0(t,s)d_{\alpha}s \\
 & & -\frac{3}{2}\sin(h_1(t,0)) \int_t^{\pi^*} \cos_{\alpha}(1;s,0)\sin(2h_1(s,0))e_0(t,s)d_{\alpha}s \\
 & & + \frac{3}{2}\cos(h_1(t,0)) \int_t^{\pi^*} \sin_{\alpha}(1;s,0)\sin(2h_1(s,0))e_0(t,s)d_{\alpha}s \\
 &=& \sin(h_1(t,0))\left[1-\cos^3(h_1(t,0))\right] - \sin(h_1(t,0))\left[1+\cos^3(h_1(t,0))\right] \\
 & & -\cos(h_1(t,0))\sin^3(h_1(t,0)) - \cos(h_1(t,0))\sin^3(h_1(t,0)) \\
 &=& -2\sin(h_1(t,0))\cos(h_1(t,0))\left[\cos^2(h_1(t,0))+\sin^2(h_1(t,0))\right] \\
 &=& -2\sin(h_1(t,0))\cos(h_1(t,0))e^2_0(t,0) \\
 &=& -\sin_{\alpha}(2;t,0).
\end{eqnarray*}
Of course one could ultimately arrive at the same solution via Example \ref{coscosh} directly, together with the assumption that the solution takes the form $x=x_c+x_p$, where $x_c$ is the general solution to the corresponding homogeneous equation, $x_p$ is any particular solution, and then applying the boundary conditions. $\hfill\triangle$
\end{example}

\section*{Acknowledgements}
This research was supported by an NSF STEP grant (DUE 0969568) and by a gift to the Division of Sciences and Mathematics at Concordia College made in support of undergraduate research. I would like to thank my two first-year undergraduate students, Grace Bryan and Laura Legare, for their research efforts in the summer REU of 2016.

%%%%%%%%%%%%%%%%%%%%%%%%%%%%
%  Bibliography            %
%%%%%%%%%%%%%%%%%%%%%%%%%%%%

\end{document}